\documentclass[12pt]{amsart}
\usepackage[margin=1in]{geometry}
\usepackage{amsmath,amsthm,amssymb,amsfonts}
\usepackage[titletoc,title]{appendix}
\usepackage{mathrsfs}
\usepackage{enumerate}
\usepackage{mathrsfs}
\usepackage{comment}
\usepackage{hyperref} 
\usepackage[capitalise]{cleveref}
\usepackage{color}
\usepackage{constants}

\newtheorem{thm}{Theorem}[section]
\newtheorem{cor}[thm]{Corollary}
\newtheorem{prop}[thm]{Proposition}
\newtheorem{lem}[thm]{Lemma}

\newtheorem{quest*}{Question}
\newtheorem{prob*}{Problem}

\theoremstyle{definition}

\newtheorem{example}{Example}

\theoremstyle{remark}

\numberwithin{equation}{section}

\crefname{figure}{Figure}{Figures}
\theoremstyle{plain}
\newtheorem*{thm*}{Theorem}
\crefname{thm}{Theorem}{Theorems}
\crefname{cor}{Corollary}{Corollarys}
\newtheorem*{cor*}{Corollary}
\crefname{cor*}{Corollary}{Corollarys}
\crefname{lem}{Lemma}{Lemmas}
\crefname{prop}{Proposition}{Propositions}
\crefname{conj}{Conjecture}{Conjectures}
\newtheorem*{conj*}{Conjecture}
\crefname{conj*}{Conjecture}{Conjectures}
\crefname{defn}{Definition}{Definitions}

\newcommand{\Z}{\mathbb{Z}}
\newcommand{\R}{\mathbb{R}}
\newcommand{\Q}{\mathbb{Q}}

\newcommand{\re}{\textup{Re}}

\newcommand{\GL}{\text{GL}}

\renewcommand{\bar}{\overline}

\renewcommand{\C}{\mathbb{C}}

\renewcommand{\pmod}[1]{\, (\mathrm{mod} {\, #1})}

\renewcommand{\Re}{\mathrm{Re}}

\makeatletter
\let\@wraptoccontribs\wraptoccontribs
\makeatother

\renewcommand{\pmod}[1]{\left(\mathrm{mod}\,\,#1\right)}

\newconstantfamily{abcon}{symbol=c}

\makeatletter
\DeclareFontFamily{U}  {MnSymbolF}{}
\DeclareSymbolFont{symbolsMN}{U}{MnSymbolF}{m}{n}
\SetSymbolFont{symbolsMN}{bold}{U}{MnSymbolF}{b}{n}
\DeclareFontShape{U}{MnSymbolF}{m}{n}{
    <-6>  MnSymbolF5
   <6-7>  MnSymbolF6
   <7-8>  MnSymbolF7
   <8-9>  MnSymbolF8
   <9-10> MnSymbolF9
  <10-12> MnSymbolF10
  <12->   MnSymbolF12}{}
\DeclareFontShape{U}{MnSymbolF}{b}{n}{
    <-6>  MnSymbolF-Bold5
   <6-7>  MnSymbolF-Bold6
   <7-8>  MnSymbolF-Bold7
   <8-9>  MnSymbolF-Bold8
   <9-10> MnSymbolF-Bold9
  <10-12> MnSymbolF-Bold10
  <12->   MnSymbolF-Bold12}{}
\DeclareMathSymbol{\tbigtimes}{\mathop}{symbolsMN}{2}
\newcommand*{\bigtimes}{%
  \DOTSB
  \tbigtimes
  \slimits@ 
}
\makeatother

\title{Weak subconvexity without a Ramanujan hypothesis}
\author{Kannan Soundararajan}
\address{Department of Mathematics, Stanford University, Stanford, CA 94305}
\email{ksound@stanford.edu}
\author{Jesse Thorner}
\address{Department of Mathematics, Stanford University, Stanford, CA 94305}
\email{jthorner@stanford.edu}

\contrib[with an appendix by]{Farrell Brumley}

\date{\today}
\begin{document}
\maketitle
\section{Statement of results}
\label{sec:results}

\noindent In \cite{Sound_weak}, the first author obtained a weak subconvexity result bounding central values of a large class of $L$-functions, assuming a weak Ramanujan hypothesis on the size of Dirichlet series coefficients of the $L$-function.   If $C$ denotes the analytic conductor of the $L$-function in question, then $C^{\frac 14}$ is the size of the convexity bound, and the weak subconvexity bound established there was of the form $C^{\frac 14}/(\log C)^{1-\epsilon}$.   In this paper we establish a weak subconvexity bound of the shape $C^{\frac 14}/(\log C)^{\delta}$ for some small $\delta> 0$, but with a much milder hypothesis on the size of the Dirichlet series coefficients.   In particular our results will apply to all automorphic $L$-functions, and (with mild restrictions) to the Rankin-Selberg $L$-functions attached to two automorphic representations.

In order to make clear the scope and limitations of our results, we axiomatize the properties of $L$-functions that we need.   In Section 2 we shall discuss how automorphic $L$-functions and Rankin-Selberg $L$-functions fit into this framework.   Let $m\geq 1$ be a natural number.  We now describe axiomatically a class of $L$-functions, which  we shall denote by ${\mathcal S}(m)$.

 {\bf  1.  Dirichlet series and Euler product.}   The functions $L(s,\pi)$ appearing in the class ${\mathcal S}(m)$ will be given by a Dirichlet series and Euler product
\begin{equation}
	\label{1.1}
	L(s,\pi)=\sum_{n=1}^{\infty}\frac{a_{\pi}(n)}{n^s}=\prod_{p}L_p(s,\pi),\qquad L_p(s,\pi)=\prod_{j=1}^{m}\Big(1-\frac{\alpha_{j,\pi}(p)}{p^s}\Big)^{-1}=\sum_{j=0}^{\infty}\frac{a_{\pi}(p^j)}{p^{js}},
\end{equation}
with both the series and the product converging absolutely for $\re(s)>1$.   It will also be convenient for us to write 
\begin{equation} 
\label{1.2} 
\log L_p(s,\pi) = \sum_{k=1}^{\infty} \frac{\lambda_{\pi}(p^k)}{k p^{ks}},  \text{ where }  \lambda_\pi (p^k) = \sum_{j=1}^{m} \alpha_{j,\pi}(p)^k.  
\end{equation} 
Setting $\lambda_\pi(n)=0$ if $n$ is not a prime power, we have 
\begin{equation} 
\label{1.3} 
-\frac{L^{\prime}}{L}(s,\pi ) = \sum_{n=1}^{\infty} \frac{\lambda_\pi(n)\Lambda(n)}{n^s},  \text{ and } \log L(s,\pi)  = \sum_{n=2}^{\infty} \frac{\lambda_\pi(n) \Lambda(n)}{n^s \log n}. 
\end{equation}

{\bf 2.  Functional equation.}  Write
\begin{equation}
	\label{1.4}
	L_{\infty}(s,\pi)=N_{\pi}^{s/2}\pi^{-ms/2}\prod_{j=1}^{m}\Gamma\Big(\frac{s+\mu_{\pi}(j)}{2}\Big),
\end{equation}
where $N_{\pi} \ge 1$ is known as the ``conductor" of the $L$-function and the $\mu_{\pi}(j)$ are complex numbers.  We suppose that there is an integer $0\le r =r_{\pi} \le m$ such that 
the completed $L$-function $s^{r }(1-s)^{r }L(s,\pi)L_{\infty}(s,\pi)$ extends to an entire function of order $1$, and 
satisifies the functional equation
\begin{equation}
	\label{1.5}
s^r(1-s)^r	L(s,\pi)L_{\infty}(s,\pi)=\kappa_{\pi} s^r (1-s)^r L(1-s,\widetilde{\pi})L_{\infty}(1-s,\widetilde{\pi}).  
\end{equation}
Here $\kappa_{\pi}$ is a complex number with $|\kappa_{\pi}|=1$, and 
\begin{equation}
\label{1.6}
	L(s,\widetilde{\pi})=\sum_{n=1}^{\infty}\frac{\overline{a_{\pi}(n)}}{n^s},\qquad L_{\infty}(s,\widetilde{\pi})=N_{\pi}^{s/2}\pi^{-ms/2}\prod_{j=1}^{m}\Gamma\Big(\frac{s+\overline{\mu_{\pi}(j)}}{2}\Big).
\end{equation}
We suppose that $r$ has been chosen such that the completed $L$-function 
does not vanish at $s=1$ and $s=0$.   Thus, if $L(s,\pi)$ has a pole at $s=1$ then we are assuming that the order of this pole is at most $m$, and $r$ is taken to be the order of the pole.  If $L(s,\pi)$ has no pole at $s=1$, then we take $r=0$ and are making the assumption that the $L(1,\pi) \neq 0$.     In our work, a key measure of the ``complexity" of the $L$-function $L(s,\pi)$ is the ``analytic conductor" 
which is defined to be  
 \begin{equation}
	\label{1.7}
	C(\pi)=N_{\pi}\prod_{j=1}^{m}(1+|\mu_{\pi}(j)|).
\end{equation}

{\bf 3.  Bounds towards the generalized Ramanujan and Selberg conjectures.}   The absolute convergence of the Euler product in \eqref{1.1} implicitly 
includes the assumption that $|\alpha_{j,\pi}(p)| <p$ for all $p$ and $j$.  Further,  the Euler product shows that $L(s,\pi)$ is non-zero in Re$(s)> 1$, which 
implies that Re$(\mu_{\pi}(j)) > -1$ for all $j$ (else there would be a trivial zero of $L(s,\pi)$ in Re$(s)>1$ to compensate for a pole of $\Gamma((s+\mu_\pi(j))/2)$).   
We impose a modest strengthening of these estimates.  Namely, we assume that for all $1\le j\le m$ 
\begin{equation}
	\label{1.8}
	| \alpha_{j, \pi} (p) | \le p^{1- 1/m}, \ \  \  
	\re(\mu_{\pi}(j)) \ge -(1 - 1/m).
\end{equation}

The widely believed generalized Ramanujan and Selberg conjectures for automorphic $L$-functions state that the bounds in \eqref{1.8} hold with $1-1/m$ replaced by $0$.  While these conjectures are still open, the weak bounds in \eqref{1.8} are known both for the $L$-functions associated to automorphic representations and their Rankin-Selberg convolutions.  We could also weaken (1.8) further by replacing $1-1/m$ with $1-\delta$ for some $\delta >0$, but the present formulation is convenient and includes all $L$-functions of interest to us.

 {\bf 4.   Rankin-Selberg and Brun-Titchmarsh bounds on $\lambda_\pi(n)$.}   Our final hypothesis prescribes two mild average bounds on $|\lambda_\pi(n)|$, which can be verified by Rankin-Selberg theory for the class of $L$-functions associated to automorphic representations and their Rankin-Selberg convolutions.  First, we assume that for all $\eta>0$
\begin{equation}
 \label{1.9}
	\sum_{n=1}^{\infty}\frac{|\lambda_{\pi}(n)|\Lambda(n)}{n^{1+\eta}}\le \frac{m}{\eta}+m\log C(\pi) +O(m^2).  
\end{equation}
Second, we assume that for all $T\ge 1$ 
\begin{equation}
	\label{1.10}
	\sum_{x<n\leq xe^{1/T}}|\lambda_{\pi}(n)|\Lambda(n)\ll_{m} \frac{x}{T},\qquad \textrm{provided  }  x\gg_{m} (C(\pi)T)^{144m^3}.
\end{equation}
There is considerable latitude in formulating the conditions \eqref{1.9} and \eqref{1.10}, and for example we could have chosen the range for $x$ in \eqref{1.10} differently.  The specific choice made here is based on the applicability of these conditions to automorphic $L$-functions. When $T$ is of constant size, the criterion \eqref{1.10} may be viewed as a Chebyshev type estimate for $|\lambda(n)|$ (generalizing $\sum_{n\le x} \Lambda(n) \ll x$), while for larger $T$ the criterion \eqref{1.10} is an analogue of the classical Brun--Titchmarsh inequality
\begin{equation}
\label{eqn:BT}
	\sum_{x< n\le x + h} \Lambda(n) \ll_{\epsilon} h,\qquad \text{for all } x \ge h \ge x^{\epsilon}.
\end{equation}

We denote by ${\mathcal S}(m)$ the class of $L$-functions satisfying the properties laid out in articles 1 to 4 above; see 
\eqref{1.1}--\eqref{1.10}.    Before stating our results, we introduce the quantity
\begin{equation}
\label{1.11}
N_{\pi}(\sigma,T):=\#\{\rho=\beta+i\gamma\colon L(\rho,\pi)=0,~\beta>\sigma,~|\gamma|\leq T\}, 
\end{equation}
which arises in the study of ``zero density estimates."     

\begin{thm}
\label{thm:bound_1/2}
	If $L(s,\pi)$ is an $L$-function in the class ${\mathcal S}(m)$ and $0\leq\delta<\frac{1}{2}$, then
	\[
	\log|L(1/2,\pi)|\leq \Big(\frac{1}{4}-10^{-9}\delta\Big)\log C(\pi)+10^{-7}\delta N_{\pi}(1-\delta,6)+2\log|L(3/2,\pi)|+O(m^2).
	\]
\end{thm}

Theorem \ref{thm:bound_1/2} adds to a long line of investigations relating the size of $L$-functions to the distribution of their zeros.   For example, it is well known that the generalized Riemann hypothesis implies the generalized Lindel{\" o}f hypothesis.  One could weaken the assumption of GRH, and establish (as Backlund did originally for $\zeta(s)$) that if almost all the zeros of the $L$-function up to height $1$ are in the region Re$(s) <1/2 +\epsilon$, then the Lindel{\" o}f bound $L(1/2,\pi) \ll C(\pi)^{\epsilon}$ would follow.   In contrast, Theorem \ref{thm:bound_1/2} states that the more modest assumption that not too many of the zeros of $L(s,\pi)$ are very close to the $1$-line leads to  a subconvex bound for $L(1/2, \pi)$  (which is a modest form of the Lindel{\" o}f bound).   For recent related work in the context of character sums and zeros of Dirichlet $L$-functions, see \cite{GS}.   The proof of Theorem \ref{thm:bound_1/2} is a refinement of an argument of Heath-Brown \cite{HB} to prove sharp convexity bounds for $L$-values.

To obtain from Theorem \ref{thm:bound_1/2} a genuine subconvexity bound of the form $L(1/2, \pi) \ll C(\pi)^{\frac 14-\delta}$ for some $\delta>0$, we would need a zero density estimate of the form $N_{\pi} (1-\delta, 6) \le 10^{-4} \log C(\pi)$, which we are unable to establish for any fixed $\delta>0$.   However, one can establish a ``log-free" zero density estimate which will permit values of $\delta$ of size $(\log \log C(\pi))/\log C(\pi)$.   This will then lead to the weak subconvexity bound where a power of $\log C(\pi)$ is saved over the convexity bound.


\begin{thm}
	\label{thm:LFZDE}
{Let $L(s,\pi)\in\mathcal{S}(m)$ and $T\geq 1$.  For all $1/2 \leq\sigma\leq 1$,
\[
N_{\pi}(\sigma,T)\ll_{m}(C(\pi)T)^{10^7 m^3(1-\sigma)}.
\]
}
\end{thm}

Log-free zero density estimates have a long history, going back to Linnik's pioneering work on the least prime in arithmetic progressions.  Our proof of Theorem \ref{thm:LFZDE} follows an argument of Gallagher, based on Tur{\' a}n's power sum method.  A key feature is the formulation of hypotheses \eqref{1.9} and \eqref{1.10}, which are $L^1$-bounds that can be verified for $L$-functions associated to automorphic representations and their Rankin-Selberg convolutions.   Thus Theorem \ref{thm:LFZDE} applies to a larger class of $L$-functions than the earlier log free zero-density estimates established by  (for example) Kowalski and Michel \cite{KM}, Motohashi \cite{Motohashi}, Akbary and Trudgian \cite{AT}, and Lemke Oliver and Thorner \cite{RJLOT}.   We have not made any attempt to optimize the exponent $10^7 m^3$, but our argument does not seem to yield an exponent independent of $m$.

Combining \cref{thm:bound_1/2,thm:LFZDE}, we deduce the following bound for $L(1/2,\pi)$.

\begin{cor}
	\label{thm:weak_subconvexity}
	Let $L(s,\pi)\in\mathcal{S}(m)$.  Then 
	\[
	|L(1/2,\pi)|\ll_{m} |L(3/2,\pi)|^2\frac{C(\pi)^{1/4}}{(\log C(\pi))^{1/(10^{17}m^3)}}.
	\]
\end{cor}

In the above corollary, one should expect the term $|L(3/2,\pi)|$ (which is evaluated in the region of absolute convergence) to be bounded, in which case the corollary furnishes a weak subconvexity bound.   The boundedness of $|L(3/2,\pi)|$ would follow for example from a stronger version of assumption \eqref{1.8}, and we shall check that this holds for automorphic $L$-functions.   For Rankin-Selberg convolutions of automorphic representations, we cannot give a satisfactory bound for the $L$-value at $3/2$ in complete generality.   Compared to the work in \cite{Sound_weak},  Corollary \ref{thm:weak_subconvexity} extends considerably the class of $L$-functions for which a weak subconvexity bound may be established, but the power of $\log C(\pi)$ saved is smaller than in \cite{Sound_weak}.

\subsection*{Acknowledgments}   We thank Dimitris Koukoulopoulos, James Maynard, and Maksym Radziwi{\l \l} for many helpful conversations.  We are particularly grateful to Paul Nelson and Dinakar Ramakrishnan for discussions regarding Lemma \ref{lem:AC_bounds} including pointing out \cite{BH2}, and to Farrell Brumley for supplying the Appendix.  Finally, we thank the anonymous referees for their helpful comments.  Kannan Soundararajan was partially supported by NSF grant DMS 1500237, and a Simons Investigator grant from the Simons Foundation.  Jesse Thorner was partially supported by an NSF Postdoctoral Fellowship.  Part of this work was carried out when the authors were in residence at MSRI, Berkeley during the Spring semester of 2017, supported in part by NSF grant DMS 1440140.

\section{Applications to automorphic $L$-functions}
\label{sec:auto}

\noindent In this section we describe how the framework and results described in Section 1 apply to automorphic $L$-functions.   We restrict attention to automorphic representations over ${\Bbb Q}$, and let ${\mathcal A}(m)$ denote the set of all cuspidal automorphic representations of $GL_m$ over ${\Bbb Q}$ with unitary central character.   Here we give a brief description of the analytic properties of the standard $L$-functions associated to such automorphic representations.  Our goal is twofold:  we wish to show that elements of ${\mathcal A}(m)$ give rise to $L$-functions in the class ${\mathcal S}(m)$, and also that if $\pi_1 \in {\mathcal A}(m_1)$ and $\pi_2\in {\mathcal A}(m_2)$ then the Rankin--Selberg $L$-function $L(s, \pi_1 \times \pi_2)$ fits into the framework of ${\mathcal S}(m_1 m_2)$.   For proofs and further discussion of the properties that we need we refer to  \cite{GJ2,JPSS,MW}, or the surveys in Michel \cite[Lecture 1]{Michel}, or Brumley \cite[Section 1]{Brumley}.

Properties 1 to 3 listed in Section 1 follow from the standard theory of automorphic forms, while Property 4 will require further discussion.   Thus, given $\pi \in {\mathcal A}(m)$, its standard $L$-function $L(s,\pi)$ has a Dirichlet series,  Euler product, and satisfies a functional equation, exactly as described in \eqref{1.1} to \eqref{1.6}.   Note also that here $\widetilde{\pi}$ denotes the contragredient representation of $\pi$.   Concerning Property 3, for $\pi \in {\mathcal A}(m)$ it is known that 
\begin{equation} 
\label{2.1}
|\alpha_{j, \pi}(p)| \le p^{\theta_m}, \qquad |\Re(\mu_{\pi}(j))| \le \theta_m, 
\end{equation} 
where 
\begin{equation} 
\label{2.2} 
\theta_m = \begin{cases}
	0&\mbox{if $m=1$,}\\
	7/64&\mbox{if $m=2$,}\\
	1/2-1/(m^2+1)&\mbox{if $m\geq 3$.}
\end{cases}
\end{equation} 
The bounds follow from the work of Luo, Rudnick, and Sarnak \cite{LRS} for $m\geq 3$ and Kim and Sarnak \cite[Appendix 2]{Kim} for $m=2$ in the unramified cases.  The ramified cases are handled by M{\"u}ller and Speh \cite[Proposition 3.3]{MS} for $m\geq 3$ and Brumley and Blomer \cite{BB} for $m=2$.  The generalized Ramanujan and Selberg conjectures assert that $\theta_m$ may be taken as $0$ in \eqref{2.1}.
 
 Now we turn to Rankin--Selberg $L$-functions.   If $\pi_1 \in {\mathcal A}(m_1)$ and $\pi_2 \in {\mathcal A}(m_2)$ are two automorphic representations, then the Euler product and Dirichlet series expansions of the Rankin-Selberg $L$-function $L(s,\pi_1\times\pi_2)$ are given by
\[
L(s,\pi_1\times\pi_2)=\sum_{n=1}^{\infty}\frac{a_{\pi_1\times\pi_2}(n)}{n^s}=\prod_{p} \prod_{j_1=1}^{m_1}\prod_{j_2=1}^{m_2}\Big(1-\frac{\alpha_{j_1,j_2,\pi_1\times\pi_2}(p)}{p^s}\Big)^{-1}.
\]
Here we may index the parameters $\alpha_{j_1,j_2, \pi_1\times \pi_2}(p)$ in such a way that, for 
all $p\nmid N_{\pi_1}N_{\pi_2}$, one has 
\begin{equation}
\label{eqn:separation_coprime}
\alpha_{j_1,j_2,\pi_1\times\pi_2}(p) = \alpha_{j_1,\pi_1}(p)\alpha_{j_2,\pi_2}(p).
\end{equation}
At the archimedean place, we write 
\begin{equation}
\label{eqn:gamma_RS}
L_{\infty}(s,\pi_1\times\pi_2)=N_{\pi_1\times\pi_2}^{s/2}\pi^{-m_1 m_2 s/2}\prod_{j_1=1}^{m_1}\prod_{j_2=1}^{m_2}\Gamma\Big(\frac{s + \mu_{\pi_1\times\pi_2}(j_1,j_2)}{2}\Big).
\end{equation}
If both $\pi_1$ and $\pi_2$ are unramified at infinity, one may write 
\begin{equation}
\label{eqn:separation_unramified_infinity}
\mu_{\pi_1\times\pi_2}(j_1,j_2)= \mu_{\pi_1}(j_1)+\mu_{\pi_2}(j_2).
\end{equation}
See \cref{lem:AC_bounds} below for a complete description of $\mu_{\pi_1\times\pi_2}(j_1,j_2)$ in the general case.   As part of the Langlands functoriality conjectures, one expects that $\pi_1 \times \pi_2$ corresponds to an automorphic representation of $GL(m_1m_2)$ (not necessarily cuspidal), but this remains unknown, apart from the work of Ramakrishnan \cite{Ramakrishnan} in the case $m_1=m_2=2$ and the work of Kim and Shahidi \cite{KimShahidi} in the case $m_1=2$ and $m_2 =3$.
 
Properties 1 and 2 may thus be verified for Rankin-Selberg $L$-functions.   As for Property 3, using \eqref{2.1} and \eqref{2.2}, and proceeding as in  \cite[Appendix]{RS} (see also \cite[Section 3]{BB2} and \cite[Section 1]{Brumley}), we obtain for all primes $p$ 
\begin{equation}
\label{eqn:LRS_auto_2}
|\alpha_{j_1,j_2,\pi_1\times\pi_2}(p)|\leq p^{\theta_{m_1}+\theta_{m_2}},\qquad \Re(\mu_{\pi_1\times\pi_2}(j_1,j_2))\geq -\theta_{m_1}-\theta_{m_2}. 
\end{equation}
The reader may also consult the explicit description of $L_{\infty}$ given in \cref{lem:AC_bounds} below, and the explicit description of $L_p$ given 
by \eqref{expression} in the Appendix.

 So far we have discussed how automorphic $L$-functions and Rankin-Selberg $L$-functions satisfy Properties 1 to 3 of Section 1.  To facilitate our discussion of Property 4, we require two lemmas.

 \begin{lem}
\label{lem:AC_bounds}
If $\pi_1\in\mathcal{A}(m_1)$ and $\pi_2\in\mathcal{A}(m_2)$, then
\[
C(\pi_1\times\pi_2)\leq e^{O(m_1 m_2)}C(\pi_1)^{m_2}C(\pi_2)^{m_1}, 
\]
and 
\[ 
C(\pi_1\times\widetilde{\pi}_1)^{m_2^2}C(\pi_2\times\widetilde{\pi}_2)^{m_1^2} \leq e^{O((m_1 m_2)^2)} C(\pi_1\times\pi_2)^{4m_1 m_2}.
\]
\end{lem}
\begin{proof}
We write $\pi\in\mathcal{A}(m)$ and $\pi'\in\mathcal{A}(m')$ instead of $\pi_1\in\mathcal{A}(m_1)$ and $\pi_2\in\mathcal{A}(m_2)$ to avoid having too many subscripts.  Let
\begin{equation}
\label{eqn:K_def2}
K_{\pi}=\prod_{j=1}^{m}(1+|\mu_{\pi}(j)|)
\end{equation}
so that $C(\pi)=N_{\pi}K_{\pi}$.  For a prime $p$, let $\mathrm{ord}_p(N_{\pi})$ be the exponent of $p$ in the prime factorization of $N_{\pi}$; in particular, $\mathrm{ord}_p(N_{\pi})=0$ if and only if $p\nmid N_{\pi}$.  Bushnell and Henniart proved that (see \cite[Theorem 1]{BH} or \cite[Corollary C]{BH2}) 
	\begin{equation}
	\label{eqn:BH_upper}
	\mathrm{ord}_p(N_{\pi \times\pi'})\leq m'\cdot\mathrm{ord}_p(N_{\pi })+m\cdot\mathrm{ord}_p(N_{\pi'})-\min\{\mathrm{ord}_p(N_{\pi }),\mathrm{ord}_p(N_{\pi'})\}, 
	\end{equation}
	and also that (see \cite[Corollary B]{BH2})
	\[
	\frac{\mathrm{ord}_p(N_{\pi \times\pi'})}{m' m}\geq\frac{1}{2}\Big(\frac{\mathrm{ord}_p(N_{\pi \times\widetilde{\pi}})}{m^2}+\frac{\mathrm{ord}_p(N_{\pi'\times\widetilde{\pi}'})}{(m')^2}\Big).
	\]
	These bounds imply
	\[
	N_{\pi \times\pi'}\leq \frac{N_{\pi }^{m'}N_{\pi'}^{m}}{\gcd(N_{\pi },N_{\pi'})}\qquad\textup{and}\qquad N_{\pi \times\pi'}\geq N_{\pi \times\widetilde{\pi}}^{\frac{m'}{2m}}N_{\pi'\times\widetilde{\pi}'}^{\frac{m}{2m'}}.
	\]
	
	We require corresponding bounds for $K_{\pi \times\pi'}$.  Brumley \cite[Appendix]{Humphries} proved that
	\[
	K_{\pi \times\pi'}\leq e^{O(m'm)}K_{\pi }^{m'}K_{\pi'}^{m}.
	\]
	It remains to establish the bound
	\begin{equation}
	\label{eqn:K_inequality2}
	K_{\pi \times\widetilde{\pi}}K_{\pi'\times\widetilde{\pi}'}\leq e^{O((m'm)^2)}K_{\pi \times\pi'}^{4m'm}.
	\end{equation}
	In order to prove \eqref{eqn:K_inequality2} regardless of the ramification at infinity, we use the archimedean case of the local Langlands correspondence as described by M{\"u}ller and Speh \cite[Proof of Lemma 3.1, $F=\R$]{MS}.  We give a brief account of the archimedean factors.  Langlands proved that there exist collections of irreducible representations $\{\varphi_i\}_{i\in\mathcal{I}}$ and $\{\varphi_{j}'\}_{j\in\mathcal{J}}$ of the Weil group $W_{\R}$ such that $\pi_{\infty}$ and $\pi_{\infty}'$ correspond to the direct sums $\oplus_{i\in\mathcal{I}}\varphi_i$ and $\oplus_{j\in\mathcal{J}}\varphi_j'$, respectively.  Each irreducible representation $\varphi$ of $W_{\R}$ is of dimension $1$ or $2$; furthermore, one has the 
	factorizations 
	\begin{align*}
		L_{\infty}(s,\pi)=\prod_{i\in \mathcal{I}}L(s,\varphi_{i}),\quad L_{\infty}(s,\pi')=\prod_{j\in \mathcal{J}}L(s,\varphi_{j}'),\quad L_{\infty}(s,\pi\times\pi')=\prod_{\substack{i\in\mathcal{I} \\ j\in\mathcal{J}}}L(s,\varphi_{i}\otimes\varphi_{j}').
	\end{align*}

To describe further the $L$-functions above, it is convenient to define $\Gamma_{\R}(s)=\pi^{-s/2}\Gamma(s/2)$ and $\Gamma_{\C}(s)=\Gamma_{\R}(s)\Gamma_{\R}(s+1) =2(2\pi)^{-s} \Gamma(s)$.
	\begin{enumerate}
		\item If $\varphi$ is one-dimensional, then there exist $\nu\in\C$ and $\varepsilon\in\{0,1\}$ such that
	\[
	L(s,\varphi)=\Gamma_{\R}(s+\nu+\varepsilon).
	\]
	We define $K(\varphi)=1+|\nu+\varepsilon|$.
	\item  If $\varphi$ is two-dimensional, then there exist $k\in\Z$ and $\nu\in\mathbb{C}$ such that
	\[
	L(s,\varphi)=\Gamma_{\mathbb{C}}(s+\nu+|k|/2).
	\]
	We define $K(\varphi)=(1+|\nu+|k|/2|)(1+|\nu+|k|/2+1|)$.
	\end{enumerate}
In both of the above cases, Rudnick and Sarnak \cite[Appendix A.3]{RS} proved that
	\begin{equation}
	\label{eqn:nu_RS}
		|\re(\nu)|<1/2.
	\end{equation}
	
M{\"u}ller and Speh also describe the $L$-functions associated to the tensor products $\varphi\otimes\varphi'$.
	\begin{enumerate}
		\item If both $\varphi$ and $\varphi'$ are one-dimensional, then $\varphi\otimes\varphi'$ is one-dimensional and
		\begin{align*}
		L(s,\varphi\otimes\varphi')=\Gamma_{\R}(s+\nu+\nu'+\varepsilon_{\varphi\otimes\varphi'}),
		\end{align*}
		where $\varepsilon_{\varphi\otimes\varphi'}\in\{0,1\}$ and $\varepsilon_{\varphi\otimes\varphi'}\equiv \varepsilon+\varepsilon'\pmod{2}$.  In this case, we define $K(\varphi\otimes\varphi')=1+|\nu+\nu'+\varepsilon_{\varphi\otimes\varphi'}|$.
		\item If $\varphi$ is one-dimensional and $\varphi'$ is two-dimensional, then $\varphi\otimes\varphi'$ is two-dimensional and
		\begin{align*}
		L(s,\varphi\otimes\varphi')=\Gamma_{\mathbb{C}}(s+\nu+\nu'+|k'|/2).
		\end{align*}
		In this case, we define $K(\varphi\otimes\varphi')=(1+|\nu+\nu'+|k'|/2|)(1+|\nu+\nu'+|k'|/2+1|)$
		\item Suppose that $\varphi$ and $\varphi'$ are two-dimensional. Then $\varphi\otimes\varphi'$ is the direct sum of two two-dimensional representations and
		\begin{align*}
		L(s,\varphi\otimes\varphi')=\Gamma_{\mathbb{C}}(s+\nu+\nu'+|k+k'|/2)\Gamma_{\mathbb{C}}(s+\nu+\nu'+|k-k'|/2).
		\end{align*}
		In this case, we define
		\begin{align*}
			K(\varphi\otimes\varphi')=~&(1+|\nu+\nu'+|k+k'|/2|)(1+|\nu+\nu'+|k+k'|/2+1|)\\
			&\times(1+|\nu+\nu'+|k-k'|/2|)(1+|\nu+\nu'+|k-k'|/2+1|).
		\end{align*}
	\end{enumerate}
These definitions give us a complete description of
	\begin{equation}
		\label{eqn:K_prod}
K_{\pi\times\pi'}=\prod_{\substack{i\in\mathcal{I} \\ j\in\mathcal{J}}}K(\varphi_{i}\otimes\varphi_{j}').
	\end{equation}

We now address \eqref{eqn:K_inequality2}.  First, assume that both $\pi $ and $\pi'$ are unramified at infinity, in which case \eqref{eqn:separation_unramified_infinity} holds.   Suppose $z_1$, $z_2$, $w_1$ and $w_2$ are complex numbers all having 
real part $\ge -1/2$.  We claim that 
	\begin{equation}
		\label{eqn:inequality}
		\frac{(1+|z_1+\overline{w_1}|)(1+|z_2+\overline{w_2}|)}{(1+|z_1+z_2|)(1+|z_1+w_2|)(1+|w_1+z_2|)(1+|w_1+w_2|)}\leq C
	\end{equation}
for some absolute constant $C$.  The triangle inequality gives 
$$ 
1+ |z_1 +\overline{w_1}| = |1+ z_1+z_2 - (z_2 -\overline{w_1})| \le 1+ |z_1+ z_2| + |z_2 -\overline{w_1}| \ll 1+ |z_1+ z_2| + |z_2 +w_1|, 
$$ 
where the last estimate follows because the real parts of $z_1$ and $w_1$ are both bounded below by $-1/2$ so that $|z_2 -\overline{w_1}| \le O(1) + |z_2+w_1|$.   In the same way 
one sees that $1+ |z_1+ \overline{w_1}| \ll 1+ |w_1+w_2| + |z_1+w_2|$, and two similar inequalities for $1+|z_2  +\overline{w_2}|$ hold.  Multiplying these four estimates 
together and taking square-roots yields \eqref{eqn:inequality}. 

Apply \eqref{eqn:inequality} with $z_1 = \mu_{\pi }(i_1)$, $w_1= \mu_{\pi }(i_2)$ and $z_2 = \mu_{\pi'}(j_1)$, $w_2= \mu_{\pi'}(j_2)$, where $1\le i_1, i_2 \le m$ and $1 \le j_1, j_2 \le m'$.   Taking the product over all the inequalities so obtained, we arrive at \eqref{eqn:K_inequality2} in the case when both $\pi $ and $\pi'$ are unramified at infinity.
	
If at least one of $\pi $ and $\pi'$ is ramified at infinity, then by \eqref{eqn:K_prod}, the bound \eqref{eqn:K_inequality2} is equivalent to the bound
	\begin{equation}
	\label{eqn:quotient2}
	\prod_{\substack{i_1\in\mathcal{I} \\ j_1\in\mathcal{J}}}\prod_{\substack{i_2\in\mathcal{I} \\ j_2\in\mathcal{J}}}\frac{K(\varphi_{i_1}\otimes\widetilde{\varphi}_{i_2})K(\varphi_{j_1}'\otimes\widetilde{\varphi}_{j_2}')}{K(\varphi_{i_1}\otimes\varphi_{j_1}')K(\varphi_{i_1}\otimes\varphi_{j_2}')K(\varphi_{i_2}\otimes\varphi_{j_1}')K(\varphi_{i_2}\otimes\varphi_{j_2}')}\ll e^{O((m'm)^2)}.
	\end{equation}
If each of $\varphi_{i_1}$, $\varphi_{i_2}$, $\varphi_{j_1}'$, and $\varphi_{j_2}'$ is one-dimensional, then we are led to the quotient
{\small\[
\frac{(1+|\nu_{i_1}+\overline{\nu_{i_2}}+\varepsilon_{\varphi_{i_1}\otimes\widetilde{\varphi}_{i_2}}|)(1+|\nu_{j_1}'+\overline{\nu_{j_2}'}+\varepsilon_{\varphi_{j_1}'\otimes\widetilde{\varphi}_{j_2}'}|)}{(1+|\nu_{i_1}+\nu_{j_1}'+\varepsilon_{\varphi_{i_1}\otimes\varphi_{j_1}'}|)(1+|\nu_{i_1}+\nu_{j_2}'+\varepsilon_{\varphi_{i_1}\otimes\varphi_{j_2}'}|)(1+|\nu_{i_2}+\nu_{j_1}'+\varepsilon_{\varphi_{i_2}\otimes\varphi_{j_1}'}|)(1+|\nu_{i_2}+\nu_{j_2}'+\varepsilon_{\varphi_{i_2}\otimes\varphi_{j_2}'}|)}.
\]}%
Recall that $\varepsilon_{\varphi\otimes\varphi'}\in\{0,1\}$.  In light of \eqref{eqn:nu_RS}, this quotient is a mild perturbation of \eqref{eqn:inequality}, and we conclude that it is absolutely bounded.  Proceeding similarly for the other cases, we observe that the product in \eqref{eqn:quotient2} is a product of mild perturbations of \eqref{eqn:inequality}, each of which is absolutely bounded.  This proves \eqref{eqn:quotient2}.
\end{proof}

\begin{lem} 
\label{lemFB}
Let $\pi_1 \in {\mathcal A}(m_1)$ and $\pi_2 \in {\mathcal A}(m_2)$.  
With the notation 
\[
\log L(s, \pi_1 \times \pi_2) = \sum_{n=2}^{\infty} \frac{\lambda_{\pi_1 \times \pi_2} (n) \Lambda(n)}{n^s \log n}, 
\]
for all prime powers $n$ we have 
\[
|\lambda_{\pi_1 \times \pi_2} (n)| \le \sqrt{\lambda_{\pi_1 \times {\widetilde \pi_1}}(n) \lambda_{\pi_2 \times \widetilde{\pi}_2}(n)} \le 
 \frac 12 \Big( \lambda_{\pi_1 \times \widetilde{\pi}_1}(n) + \lambda_{\pi_2 \times \widetilde{\pi}_2}(n) \Big). 
\]
Further, for any $\pi \in {\mathcal A}(m)$ we have 
\[
|\lambda_\pi(n)| \le \sqrt{\lambda_{\pi \times \widetilde{\pi}}(n)} \le \frac 12 \Big( 1 +\lambda_{\pi \times \widetilde{\pi}} (n) \Big). 
\]
\end{lem}

 If $n$ is the power of an unramified prime $p$, then from \eqref{eqn:separation_coprime} one may see that $\lambda_{\pi_1 \times \pi_2}(n)  = \lambda_{\pi_1}(n) \lambda_{\pi_2}(n)$, and that $\lambda_{\pi_1 \times \widetilde{\pi}_1}(n) = |\lambda_{\pi_1}(n)|^2$ and $\lambda_{\pi_2 \times \widetilde{\pi}_2}(n) = |\lambda_{\pi_2}(n)|^2$.   In this situation, the bound of Lemma \ref{lemFB} follows readily by Cauchy--Schwarz.   The point of the lemma is that the same bound applies in the ramified case also.  We thank Farrell Brumley for supplying a proof of this fact in \cref{sec:appendix}.

We now discuss Property 4 with relation to automorphic $L$-functions, starting with the estimate \eqref{1.9}.  In the next section, we shall establish the following lemma, from which we can deduce \eqref{1.9}.      

\begin{lem}
\label{lem2.2}   If $\pi \in {\mathcal A}(m)$ is a cuspidal automorphic representation then for any $\eta >0$ 
\begin{equation} 
\label{2.6}  
\sum_{n=1}^{\infty} \frac{\lambda_{\pi \times \widetilde{\pi}}(n) \Lambda(n)} {n^{1+\eta}} \le \frac 1\eta +  \frac12 \log C(\pi \times \widetilde{\pi}) + O(m^2). 
\end{equation} 
\end{lem} 

\begin{proof}[Verifying \eqref{1.9} for $\pi \in {\mathcal A}(m)$]   Applying Lemma \ref{lemFB}, we find that 
\[
\sum_{n=1}^{\infty} \frac{|\lambda_{\pi}(n)| \Lambda(n)}{n^{1+\eta}} \le \frac 12\Big( \sum_{n=1}^{\infty} (1+\lambda_{\pi \times{\widetilde{\pi}}}(n)) \frac{\Lambda(n)}{n^{1+\eta}}\Big) \le \frac{1}{\eta} + \frac{1}{4} \log C(\pi \times{\widetilde{\pi}}) + O(m^2), 
\]
by Lemma \ref{lem2.2}.   Now applying Lemma \ref{lem:AC_bounds}, we see that $\log C(\pi \times \widetilde{\pi}) \le 2m \log C(\pi)$, and therefore 
\begin{equation*} 
\sum_{n= 1}^{\infty} \frac{|\lambda_\pi (n)| \Lambda(n)}{n^{1+\eta}}  \le \frac 1\eta + m \log C(\pi) + O(m^2). 
\end{equation*} 
This verifies \eqref{1.9} for cuspidal automorphic representations.   
\end{proof} 

\begin{proof}[Verifying \eqref{1.9} for $\pi_1 \times \pi_2$]  If $\pi_1 \in {\mathcal A}(m_1)$ and $\pi_2 \in {\mathcal A}(m_2)$ are two cuspidal automorphic representations, then from Lemma \ref{lemFB} and Lemma \ref{lem2.2} we see that 
 \begin{align*}
	\sum_{n=1}^{\infty}\frac{|\lambda_{\pi_1\times\pi_2}(n)|\Lambda(n)}{n^{1+\eta}}&\le \frac 12 \sum_{n=1}^{\infty} (\lambda_{\pi_1\times\widetilde{\pi}_1}(n)+\lambda_{\pi_2\times\widetilde{\pi}_2}(n)) \frac{\Lambda(n)}{n^{1+\eta}}\\
	&\le \frac{1}{\eta}+\frac{1}{4} \log C(\pi_1\times\widetilde{\pi}_1)+ \frac{1}{4} \log C(\pi_2\times\widetilde{\pi}_2)+ O(m_1^2+m_2^2).  
\end{align*}
Appealing now to Lemma \ref{lem:AC_bounds}, we conclude that for any $\eta>0$ 
\begin{equation*} 
\sum_{n=1}^{\infty}\frac{|\lambda_{\pi_1\times\pi_2}(n)|\Lambda(n)}{n^{1+\eta}}\le \frac{1}{\eta}+ m_1 m_2\log C(\pi_1\times\pi_2) +O((m_1m_2)^2).
\end{equation*} 
This completes our verification of \eqref{1.9} for the Rankin--Selberg convolution $\pi_1 \times \pi_2$.  
\end{proof}

In Section 6, we will prove the following theorem, from which we will deduce \eqref{1.10} for $L(s,\pi_1)$ and $L(s,\pi_1\times\pi_2)$.

\begin{thm}
\label{thm:weaker_ramanujan_1}
Let $\pi\in\mathcal{A}(m)$  be a cuspidal automorphic representation.  If 
$x\gg_{m}C(\pi\times\widetilde{\pi})^{36m^2}$ and $1\leq T\leq x^{\frac{1}{9m^2}}$, then
\[
\sum_{x<n\leq x e^{1/T}}\lambda_{\pi\times\widetilde{\pi}}(n)\Lambda(n)\ll_m \frac{x}{T}.
\]
\end{thm}

\begin{proof}[Deducing \eqref{1.10} for $L(s,\pi)$]  By Lemma \ref{lemFB} 
\begin{equation} 
\label{2.9}
\sum_{x< n\le xe^{1/T}} |\lambda_\pi(n)| \Lambda(n) \le \frac{1}{2} \sum_{x < n \le xe^{1/T}} \Big( 1+ \lambda_{\pi \times\widetilde{\pi}}(n)\Big) \Lambda(n). 
\end{equation} 
By Theorem \ref{thm:weaker_ramanujan_1}, the second term in the right side above contributes $\ll x/T$, provided $1\le T\le x^{\frac{1}{9m^2}}$ and $x\ge C(\pi \times {\widetilde{\pi}})^{36m^2}$.  In view of Lemma \ref{lem:AC_bounds}, it suffices to assume that $x\ge (C(\pi)T)^{72m^3}$.  For the same range of $x$ and $T$, the Brun-Titchmarsh inequality \eqref{eqn:BT} bounds the first term in the right side of \eqref{2.9} by $\ll x/T$, which completes our deduction.
\end{proof}  

\begin{proof}[Deducing \eqref{1.10} for $L(s,\pi_1\times \pi_2)$]  This follows similarly, appealing to Lemma \ref{lem:AC_bounds}, 
Lemma \ref{lemFB}, and Theorem \ref{thm:weaker_ramanujan_1}.  
\end{proof}

Gathering together the observations made so far, we arrive at the following proposition.  

\begin{prop}  If $\pi \in {\mathcal A}(m)$ is a cuspidal automorphic representation, then $L(s, \pi)$ is in the class ${\mathcal S}(m)$.  If $\pi_1 \in {\mathcal A}(m_1)$ and $\pi_2 \in {\mathcal A}(m_2)$ are two cuspidal automorphic representations, then $L(s,\pi_1\times \pi_2)$ is in the class ${\mathcal S}(m_1m_2)$.   
\end{prop}

Therefore the results given in \cref{sec:results} apply 
in the context of automorphic $L$-functions and yield the following corollaries.  

\begin{cor}
\label{thm:lfzde}  If
 $\pi \in {\mathcal A}(m)$ is a cuspidal automorphic representation, then for all $T\ge 1$ and $\frac 12\le \sigma \le 1$ 
we have  
\[
N_{\pi} (\sigma, T) \ll_m (C(\pi) T)^{10^{7} m^3 (1-\sigma)}. 
\]
Further, if $\pi_1 \in {\mathcal A}(m_1)$ and $\pi_2 \in {\mathcal A}(m_2)$ are two cuspidal automorphic representations, then for all $T\ge 1$ and 
$\frac 12 \le \sigma \le 1$ we have 
\[
N_{\pi_1\times\pi_2}(\sigma,T)\ll_{m_1,m_2}(C(\pi_1\times\pi_2)T)^{10^7  m_1^3 m_2^3(1-\sigma)}.
\]
\end{cor}

Apart from the exponent, this corollary gives a general result which in special situations (or with additional hypotheses) 
was given by a number of authors; see Kowalski and Michel \cite{KM}, 
 Motohashi \cite{Motohashi}, Akbary and Trudgian \cite{AT}, and Lemke Oliver and Thorner \cite{RJLOT}.

As a consequence of \cref{thm:weak_subconvexity} we obtain the following 
weak subconvexity results for automorphic $L$-functions.

\begin{cor}
\label{thm:weak_1}  If $\pi \in {\mathcal A}(m)$ is a cuspidal automorphic representation then 
\[
|L(1/2, \pi)| \ll_m \frac{C(\pi)^{\frac 14}}{(\log C(\pi))^{1/(10^{17}m^3)}}. 
\]
If $\pi_1 \in {\mathcal A}(m_1)$ and $\pi_2 \in {\mathcal A}(m_2)$ are two cuspidal automorphic representations then 
\[
|L(1/2,\pi_1\times\pi_2)|\ll_{m_1,m_2}|L(3/2,\pi_1\times\pi_2)|^2\frac{C(\pi_1\times\pi_2)^{1/4}}{(\log C(\pi_1\times\pi_2))^{1/(10^{17} m_1^3 m_2^3)}}.
\]
\end{cor}

In the first part of Corollary \ref{thm:weak_1}, we dropped the term $|L(3/2,\pi)|^2$.   This is permissible because \eqref{2.1} and \eqref{2.2} give $|\lambda_\pi(n)| \ll n^{\theta_m}$, so the bound $|L(3/2, \pi)| \ll_m 1$ follows from \eqref{1.3}. For the general Rankin-Selberg $L$-function $L(s, \pi_1 \times \pi_2)$, we are not able to obtain the bound $|L(3/2, \pi_1 \times \pi_2)| \ll 1$---without additional hypotheses, the best known bound for $|L(3/2,\pi_1\times\pi_2)|$ follows from Theorem 2 of \cite{Li}, and this is larger than any power of $\log C(\pi_1\times\pi_2)$.

Nevertheless, in a number of special situations the term $|L(3/2, \pi_1 \times \pi_2)|^2$ may be dropped, and we give a few such examples.   

\begin{example}  If either $\pi_1$ or $\pi_2$ satisfies the Ramanujan conjecture, then using \eqref{2.1} and \eqref{2.2}, we obtain $|\lambda_{\pi_1 \times \pi_2} (n)| \ll n^{\frac 12 -\delta}$ for some $\delta =\delta(m_1,m_2)>0$.  Therefore, $|L(3/2, \pi_1 \times \pi_2)| \ll_{m_1,m_2}1$ by \eqref{1.3}.
\end{example}

\begin{example}  Since $\theta_2$ may be taken as $7/64$ (see \eqref{2.2}), if $\pi_1$ and $\pi_2$ are both cuspidal automorphic forms on $GL(2)$ then $|L(3/2, \pi_1 \times \pi_2)| \ll 1$ and 
\[
|L(1/2, \pi_1 \times \pi_2)| \ll \frac{C(\pi_1 \times \pi_2)^{1/4}}{(\log C(\pi_1 \times \pi_2))^{1/10^{19}}}. 
\]
Alternatively, here we could use the work of Ramakrishnan \cite{Ramakrishnan} which shows that $\pi_1 \times \pi_2$ is an isobaric sum of cuspidal automorphic representations of dimension at most $4$, and then use our bound for each constituent.
\end{example} 

\begin{example}  If $\pi_1$ and $\pi_2$ are cuspidal automorphic representations in ${\mathcal A}(2)$, then $\mathrm{Sym}^2\pi_1$ is an automorphic representation on $GL(3)$ (by the work of Gelbart and Jacquet \cite{GJ}).  Since $\theta_2 =7/64$, we find that $|\lambda_{\mathrm{Sym}^2 \pi_1 \times \pi_2} (n)| \ll n^{21/64}$, and so $|L(3/2, \mathrm{Sym}^2 \pi_1 \times \pi_2)| \ll 1$.   Therefore, if $\mathrm{Sym}^2 \pi_1$ is cuspidal then   
\[
|L(1/2, \mathrm{Sym}^2 \pi_1 \times \pi_2)| \ll \frac{C(\mathrm{Sym}^2 \pi_1 \times \pi_2)^{1/4}}{(\log C(\mathrm{Sym}^2 \pi_1 \times \pi_2))^{1/10^{20}}}.
\]
The bound also applies when $\mathrm{Sym}^2\pi_1$ is not cuspidal, upon decomposing this and using our result for each component.  
Similarly, one can obtain 
\[
|L(1/2, \mathrm{Sym}^2 \pi_1 \times \mathrm{Sym}^2 \pi_2)| \ll \frac{C(\mathrm{Sym}^2 \pi_1 \times \mathrm{Sym}^2 \pi_2)^{1/4}}{(\log C(\mathrm{Sym}^2 \pi_1 \times \mathrm{Sym}^2 \pi_2))^{1/10^{20}}}.
\]
\end{example}

\begin{example}  If $\pi_1$ and $\pi_2$ are in ${\mathcal A}(2)$, then $\mathrm{Sym}^3 \pi_1$ is an automorphic form on $GL(4)$ by the work of Kim and Shahidi \cite{KimShahidi}.  As in Example 3, we can obtain a weak subconvexity bound for $L(1/2, \textrm{Sym}^3 \pi_1 \times \pi_2)$.
\end{example}

\begin{example}  While we have formulated our results for the $L$-values at the central point $1/2$, with trivial modifications the results apply equally to any point $1/2+ it$ on the critical line.  If $\pi_1$ in ${\mathcal A}(m_1)$ and $\pi_2$ in ${\mathcal A}(m_2)$ are 
considered fixed, then in the $t$--aspect our work gives the weak subconvexity bound 
\[
|L(1/2 + it , \pi_1 \times \pi_2)| \ll_{\pi_1, \pi_2} \frac{(2+|t|)^{m_1m_2/4}}{(\log (2+|t|))^{1/(10^{17}m_1^3 m_2^3)}}. 
\]
Here we have used the absolute convergence of $L(s,\pi_1 \times \pi_2)$ for Re$(s)>1$ (due to Jacquet, Piatetski-Shapiro, and Shalika \cite{JPSS}) to bound 
$|L(3/2+it, \pi_1 \times \pi_2)|$ by $\ll_{\pi_1, \pi_2} 1$.  
\end{example}

\section{Preliminary lemmas}

\noindent Let $L(s,\pi)\in\mathcal{S}(m)$.  Since the Euler product expansion of $L(s,\pi)$ converges absolutely and $L_{\infty}(s,\pi)\neq0$ for $\re(s)>1$, there are no zeros of $L(s,\pi)L_{\infty}(s,\pi)$ in this region.  By the functional equation, the same must be true in the region $\re(s)<0$.  Thus all of the zeros of $L(s,\pi)L_{\infty}(s,\pi)$ lie in the critical strip $0\leq\re(s)\leq1$; we call these zeros the nontrivial zeros of $L(s,\pi)$.  On the other hand, $L(s,\pi)$ might have a zero corresponding to a pole of $L_{\infty}(s,\pi)$; we call these zeros the trivial zeros of $L(s,\pi)$.  Because the Selberg eigenvalue conjecture is not yet resolved for all $L(s,\pi)$, we might have trivial zeros in the critical strip.  Unless specifically mentioned otherwise, we will always use $\rho=\beta+i\gamma$ to denote a nontrivial zero of $L(s,\pi)$.  Note that neither $0$ nor $1$ can be a non-trivial zero of $L(s,\pi)$.

By hypothesis, $s^r(1-s)^rL(s,\pi)L_{\infty}(s,\pi)$ is an entire function of order 1, and thus has a Hadamard product representation
\begin{equation}
\label{eqn:hadamard}
s^r(1-s)^rL(s,\pi)L_{\infty}(s,\pi)=e^{a_{\pi}+b_{\pi}s}\prod_{\substack{\rho}}\Big(1-\frac{s}{\rho}\Big)e^{s/\rho},
\end{equation}
where $\rho$ runs through the nontrivial zeros of $L(s,\pi)$.  By taking the logarithmic derivative of both sides of \eqref{eqn:hadamard} we see that
\begin{equation}
\label{3.2}
\sum_{\rho}\Big(\frac{1}{s-\rho}+\frac{1}{\rho}\Big)+b_{\pi}=\frac{L'}{L}(s,\pi)+\frac{L'}{L}(s,\pi_{\infty})+\frac{r}{s}+\frac{r}{s-1}.
\end{equation}
Using \eqref{1.5} and the fact that $s^r(1-s)^r L(s,\pi)L_{\infty}(s,\pi)$ is an entire function of order 1, one can prove that $\re(b_{\pi})$ equals the absolutely convergent sum $-\sum_{\rho}\re(\rho^{-1})$.  It follows that
	\begin{equation}
	\label{3.3}
	\sum_{\rho}\Re\Big(\frac{1}{s-\rho}\Big)=\Re\Big(\frac{L^\prime}{L}(s,\pi_{\infty})+\frac{L'}{L}(s,\pi)+\frac{r}{s-1}+\frac{r}{s}\Big).
	\end{equation}

\begin{lem}\label{lem2.1}  We have 
\begin{equation} 
\label{3.4} 
N_{\pi} (0, 6) = \#\{ \rho = \beta+i\gamma: \  |\gamma|\le 6\} \ge \frac 4{15} \log C(\pi ) + O(m). 
\end{equation}  
Further, for any real number $t$, and any $0< \eta \le 1$, we have 
\begin{equation} 
\label{3.5}  
\sum_{\rho} \frac{1+\eta -\beta}{|1+\eta+it -\rho |^2}  \le 2 m\log  C(\pi) + m \log (2+|t|) +2 \frac{m}{\eta} +O(m^2),  
\end{equation} 
so that 
\begin{equation} 
\label{3.6} 
\#\{ \rho: \ |\rho - (1+it)| \le \eta\} \le 10 m \eta \log C(\pi) +  5 m \eta  \log (2+|t|) +O(m^2). 
\end{equation} 
\end{lem} 
\begin{proof}  These results all follow from the Hadamard formula \eqref{3.3}.  We start with \eqref{3.5} and \eqref{3.6}.  Apply \eqref{3.3} 
with $s= 1+\eta+it$.  The left side there is 
\begin{equation} 
\label{3.7} 
\sum_{\rho} \frac{(1+\eta -\beta)}{(1+\eta -\beta)^2 +(t-\gamma)^2} \ge \frac{1}{5\eta} \#\{ \rho: \ |\rho - (1+it)| \le \eta\}. 
\end{equation} 
The right side there is 
\[
\le \frac12 \log N_\pi + \frac 12 \sum_{j=1}^{m} \Re \frac{\Gamma^{\prime}}{\Gamma} \Big( \frac{1+\eta +it +\mu_{\pi}(j)}{2} \Big) 
+ \sum_{n=1}^{\infty} \frac{|\lambda_\pi(n)| \Lambda(n)}{n^{1+\eta}} + \frac{r}{\eta} + r,  
\]
which after using \eqref{1.9}, Stirling's formula, and the bound $r\le m$  is 
\[
\le 2m\log C(\pi) + m \log (2+|t|) + 2\frac{m}{\eta} + O(m^2). 
\]
From this estimate and \eqref{3.7} we conclude \eqref{3.5} and \eqref{3.6}.

To prove \eqref{3.4},  we begin by applying \eqref{3.3} with $s= \sigma \ge 3$.  This gives 
\[
\sum_{\rho} \frac{(\sigma-\beta)}{(\sigma- \beta)^2 +\gamma^2} = \log C(\pi) + O(m) + O\Big( \sum_{n=1}^{\infty} 
\frac{|\lambda_\pi(n)| \Lambda(n)}{n^3}\Big) = \log C(\pi) + O(m).
\]
Applying the above with $\sigma =3$ and $\sigma=4$ we obtain 
\[
\sum_{\rho} \Big( \frac{(3-\beta)}{(3-\beta)^2 +\gamma^2} - \frac {13}{15} \frac{(4-\beta)}{(4-\beta)^2 +\gamma^2} \Big) 
= \frac{2}{15 }\log C(\pi )  + O(m). 
\]
A small calculation shows that when $|\gamma| > 6$ the terms on the left side above are negative, and when $|\gamma|\le 6$ 
the corresponding term is $\le 1/(3-\beta)   \le 1/2$.  From this \eqref{3.4} follows.     
\end{proof}

We end this section by establishing Lemma \ref{lem2.2}. 

 \begin{proof}[Proof of Lemma \ref{lem2.2}]   The proof is standard, based on the Hadamard factorization formula 
(see \cite[Lemma 3.5]{RJLOT}).     Rearranging the expression for the logarithmic derivative of the Hadamard factorization 
formula for $L(s, \pi \times \widetilde{\pi})$ (see \eqref{3.3}), we must bound 
	\[
	\re\Big(-\frac{L'}{L}(1+\eta,\pi\times\widetilde{\pi})\Big)=\frac{1}{\eta}+\frac{1}{1+\eta}+\Re\Big(\frac{L_{\infty}'}{L_{\infty}}(1+\eta,\pi\times\widetilde{\pi}\Big)-\sum_{\rho\neq0,1}\re\Big(\frac{1}{1+\eta-\rho}\Big),
	\]
	where $\rho=\beta+i\gamma$ runs through the zeros of $s(1-s)L(s,\pi\times\widetilde{\pi})L_{\infty}(s,\pi\times\widetilde{\pi})$.  Since $0<\beta<1$, we have
	\[
	\re\Big(\frac{1}{1+\eta-\rho}\Big)=\frac{1+\eta-\beta}{|1+\eta-\rho|^2}>0, 
	\]
so that the contribution from zeros is negative, and may be discarded.   	
	Moreover, by Stirling's formula and \eqref{1.8},
	\begin{align*}
	\Re\Big(\frac{L_{\infty}'}{L_{\infty}}(1+\eta,\pi\times\widetilde{\pi})\Big)&=-\sum_{|1+\eta+\mu_{\pi\times\widetilde{\pi}}(j)|<1}\Re\Big(\frac{1}{1+\eta+\mu_{\pi\times\widetilde{\pi}}(j)}\Big)+\frac{1}{2}\log C(\pi\times\widetilde{\pi})+O(m^2)\\
	&\leq\frac{1}{2}\log C(\pi\times\widetilde{\pi})+O(m^2).
	\end{align*}
 Therefore,
\[
\sum_{n=1}^{\infty}\frac{\lambda_{\pi\times\widetilde{\pi}}(n)\Lambda(n)}{n^{1+\eta}}\leq \frac{1}{\eta}+\frac{1}{2}\log C(\pi\times\widetilde{\pi})+O(m^2), 
\] 
completing our proof.  
\end{proof}

\section{Proof of \cref{thm:LFZDE}}

\noindent We prove the log-free zero density estimate of \cref{thm:LFZDE} by following Gallagher's treatment \cite{Gallagher2}, 
which is based on Tur{\' a}n's power sum method.   For the sake of completeness, we show that the axiomatic 
framework given in \eqref{1.1} to \eqref{1.10} is sufficient to establish such a log-free zero density estimate.

Let $k\ge 1$ be a natural number, and let $\eta$ be  a real number with $1/\log (C(\pi) T) < \eta \le 1/(200 m)$.  Let $\tau$ be a real number with $T \ge |\tau|\ge 200\eta$.   
   Differentiating \eqref{3.2} $k$ times we find, with $s= 1+\eta +i\tau$,  
\begin{align*}
 \Big( \frac{L^{\prime}}{L} (s,\pi) \Big)^{(k)} + \Big( \sum_{j=1}^{m} \frac 12 \frac{\Gamma^{\prime}}{\Gamma} \Big( \frac{s+\mu_\pi(j)}{2} \Big) \Big)^{(k)} &+ {(-1)^k}{k!} \Big( \frac{r}{s^{k+1}} + \frac{r}{(s-1)^{k+1}} \Big) \\
 &={(-1)^k}{k!} \sum_{\rho} \frac{1}{(s-\rho)^{k+1} }. 
\end{align*}    
Since Re$(\mu_{\pi}(j)) \ge -1 + 1/m$, we obtain 
\[
\frac 12 \Big( \frac{\Gamma^{\prime}}{\Gamma} \Big( \frac{s+\mu_\pi(j)}{2} \Big) \Big)^{(k)} = 
\frac{(-1)^{k+1} k!}{2^{k+1}} \sum_{n=0}^{\infty} \frac{1}{(n+(s+\mu_{\pi}(j))/2)^{k+1}} \ll m^{k+1} k!,
\]
and since $|\tau| \ge 200\eta$ and $r\le m$ clearly  
\[
{(-1)^k}{k!} \Big( \frac{r}{s^{k+1}} + \frac{r}{(s-1)^{k+1}} \Big) \ll \frac{m  k!}{(200\eta)^{k+1}}. 
\]
Thus, since $m\le 1/(200\eta)$,  
\begin{equation} 
\label{4.1} 
\frac{(-1)^k}{k!}  \Big( \frac{L^{\prime}}{L} (s,\pi) \Big)^{(k)} = O\Big ( \frac{m}{(200\eta)^{k+1}}\Big) + \sum_{\rho} \frac{1}{(s-\rho)^{k+1}}. 
\end{equation} 
Applying \eqref{3.5} we see that 
\begin{align*} 
\Big| \sum_{\substack { \rho \\ |s-\rho| \ge 200 \eta}} \frac{1}{(s-\rho)^{k+1}}\Big|  &\le 
\frac{1}{(200 \eta)^{k-1}} \sum_{\rho} \frac{1}{|s-\rho|^2} 
\le \frac{1}{(200 \eta)^{k-1}} \frac{1}{\eta} \sum_{\rho} \frac{(1+\eta -\beta)}{|s-\rho|^2} \\
&\ll \frac{1}{(200\eta)^k} \Big( m\log (C(\pi) T) + \frac m{\eta}\Big) \ll \frac{m\log (C(\pi) T)}{(200 \eta)^{k}}.
\end{align*}  
Since $\eta \ge 1/\log (C(\pi) T)$, using this estimate in \eqref{4.1} we conclude that 
\begin{equation} 
\label{4.2} 
\frac{(-1)^k}{k!}  \Big( \frac{L^{\prime}}{L} (s,\pi) \Big)^{(k)} = O \Big( \frac{m\log (C(\pi)T) }{(200 \eta)^{k}} \Big) + 
\sum_{\substack {\rho \\ |s-\rho| \le 200\eta} } \frac{1}{(s-\rho)^{k+1}}. 
\end{equation}

Equation \eqref{4.2} forms the starting point for the proof of \cref{thm:LFZDE}.   Using Tur{\' a}n's power sum method \cite{Turan}, we shall obtain 
a lower bound for the right side of \eqref{4.2} for a suitable $k$ (which will eventually be of size about $\eta \log (C(\pi)T)$), provided there is a zero 
$\rho$ with $|1+i\tau -\rho| \le \eta$.  On the  other hand, we shall bound from above the left side of \eqref{4.2} in terms of Dirichlet polynomials over prime powers.  The interplay of these bounds will yield the theorem.    We start with the lower bound, which will use the following result from Tur{\' a}n's method (see the Theorem in \cite{Turan}).  

\begin{lem}
\label{lem:turan}
Let $z_1,\ldots,z_{\nu}\in\mathbb{C}$.  If $K \geq \nu$, then there exists an integer $k\in [K,2K]$ such that $|z_1^k+\cdots+z_{\nu}^k|\geq(|z_1|/50)^k$.
\end{lem}

\begin{lem} 
\label{lem4.2} Let $\eta$ and $\tau$ be real numbers with $1/\log (C(\pi)T) < \eta \le 1/(200m)$ and  $ 200 \eta \le |\tau| \le T$.  
Suppose that $L(s,\pi)$ has a zero $\rho_0$ satisfying $|\rho_0 - (1+i\tau)| \le \eta$.   If $K> \lceil 2000 m\eta \log (C(\pi) T) +O(m^2) \rceil$, 
then  one has (recall $s=1+\eta +i\tau$)
\[
 \Big| \sum_{\substack{\rho \\ |s-\rho| \le 200\eta }} \frac{1}{(s-\rho)^{k+1}} \Big| \ge \Big( \frac{1}{100\eta}\Big)^{k+1},
\]
 for some integer $k \in [K,2K]$.  
\end{lem} 
\begin{proof}  By \eqref{3.6} we see that there are at most $2000m\eta \log (C(\pi)T) +O(m^2)$ zeros $\rho$ satisfying $|s-\rho| \le 
200\eta$.   Applying Lemma \ref{lem:turan} with $z_1$ there being $1/(s-\rho_0)$, which is $\ge 1/(2\eta)$ 
in size, the lemma follows. 
\end{proof}

We now proceed to the upper bound.

\begin{lem} 
\label{lem4.3} Let $\eta$ and $\tau$ be real numbers with $1/\log (C(\pi)T) < \eta \le 1/(200m)$ and  $ 200 \eta \le |\tau| \le T$.  
Let $K\ge 1$ be a natural number, and put $N_0 = \exp(K/(300 \eta))$ and $N_1 =\exp(40K/\eta)$.  With $s=1+\eta+i\tau$, we have for all $K\le k\le 2K$  
\[
\Big| \frac{\eta^{k+1}}{k!} \Big(\frac{L^{\prime}}{L} (s,\pi)\Big)^{(k)} \Big| \le \eta^2 \int_{N_0}^{N_1}   \Big| \sum_{N_0\le n\le u} \frac{\lambda_\pi(n)\Lambda(n)}{n^{1+i\tau}} \Big| \frac{du}{u} + 
O\Big( \frac{m\eta \log (C(\pi)T)}{(110)^k}\Big). 
\]
\end{lem} 
\begin{proof} Computing the $k$-th derivative of the Dirichlet series for $\frac{L^{\prime}}{L}(s,\pi)$, we find
\[
\Big| \frac{\eta^{k}}{k!} \Big(\frac{L^{\prime}}{L} (s,\pi)\Big)^{(k)} \Big| = \Big|\sum_{n=1}^{\infty} \frac{\lambda_{\pi}(n) \Lambda(n)}{n^{1+\eta+i\tau}} \frac{(\eta \log n)^{k}}{k!} \Big|. 
\]
Put $j_k(u) = e^{-u} u^k/k!$, and split the sum over $n$ into the ranges $n\in [N_0, N_1]$ and $n\not \in [N_0, N_1]$.  For $n\not \in [N_0,N_1]$ 
we estimate trivially using the triangle inequality, and use partial summation in the range $n\in [N_0,N_1]$.   Thus the above is 
\begin{align} 
\label{4.3}
&\le \sum_{n\not\in [N_0, N_1]} \frac{|\lambda_\pi(n)|\Lambda(n)}{n} j_k(\eta\log n) + \sum_{N_0\le n\le N_1} \frac{|\lambda_\pi(n)| \Lambda(n)}{n} j_k(\eta \log N_1) 
\nonumber  \\
&\hskip 1 in + \int_{N_0}^{N_1} \Big| \frac{d}{du} j_k(\eta \log u)\Big| \Big| \sum_{N_0\le n\le u} \frac{\lambda_\pi(n)\Lambda(n)}{n^{1+i\tau}} \Big| du. 
\end{align}
Now $|\frac{d}{du} (j_k(\eta\log u)) | = |-j_k(\eta \log u) + j_{k-1}(\eta \log u)) | (\eta /u) \le \eta/u$, and so the integral in \eqref{4.3} 
is 
\[
\le \eta \int_{N_0}^{N_1}   \Big| \sum_{N_0\le n\le u} \frac{\lambda_\pi(n)\Lambda(n)}{n^{1+i\tau}} \Big| \frac{du}{u}. 
\]

Suppose that $n\leq N_0$, in which case $\eta\log n\leq K/300$.  Since $k\in[K,2K]$ and $k! \ge (k/e)^k$, we observe that
\[
j_k(\eta\log n)=\frac{n^{-\eta}(\eta\log n)^k}{k!}\leq n^{-\eta}\Big(\frac{e\eta\log n}{k}\Big)^k\leq n^{-\eta/2}(110)^{-k}.
\]
Next suppose that $n\geq N_1$, in which case $\eta\log n\geq 40K$.  Since $e^{-u/2}u^k/k!$ is decreasing in the range $u>2k$, we 
see that for $n\ge N_1$ 
\begin{align*}
j_k(\eta\log n)=n^{-\eta/2}\frac{e^{-\frac{1}{2}\eta\log n}(\eta\log n)^{k}}{k!}&\leq n^{-\eta/2}\frac{e^{-20K}(40K)^k}{k!}\\
&\leq n^{-\eta/2}e^{-20K}\Big(\frac{40eK}{k}\Big)^k\leq n^{-\eta/2}(110)^{-k}.
\end{align*}
The last estimate also implies that for the sum over $N_0\le n\le N_1$ 
one has $j_k(\eta \log N_1) \ll n^{-\eta/2} (110)^{-k}$. Therefore the sums appearing in \eqref{4.3} are bounded by 
\[
\ll \frac{1}{(110)^{k}} \sum_{n=1}^{\infty}\frac{|\lambda_\pi(n)| \Lambda(n)}{n^{1+\eta/2}} \ll \frac{m\log (C(\pi)T) }{(110)^{k}}
\]
using \eqref{1.9}.  
\end{proof}

We now combine \cref{lem4.2,lem4.3} to prove \cref{thm:LFZDE}.

\begin{proof}[Proof of \cref{thm:LFZDE}]  We combine \cref{lem4.2,lem4.3} to detect zeros near the line $\sigma=1$.  Let $\eta$ and $\tau$ be real numbers with $1/\log (C(\pi)T) < \eta \le 1/(200m)$ and  $ 200 \eta \le |\tau| \le T$.   In keeping with \cref{lem4.2}, we suppose that 
\begin{equation}
\label{eqn:K_def}
K= 10^5 m^3\eta \log (C(\pi) T) + O(m^2)
\end{equation}
is sufficiently large, and put (as in Lemma \ref{lem4.3}) $N_0 = \exp(K/(300 \eta))$ and $N_1 =\exp(40K/\eta)$.  Suppose that $L(s,\pi)$ has a zero $\rho_0$ satisfying $|1+i\tau-\rho_0| \le \eta$.   Since $K$ satisfies \eqref{eqn:K_def} and is sufficiently large, combining \eqref{4.2} with \cref{lem4.2} we obtain, for some $k\in[K,2K]$,  
\begin{align*}
\label{4.4}
\Big|\frac{\eta^{k+1}}{k!}\Big(\frac{L'}{L}(s,\pi)\Big)^{(k)}\Big|\geq \Big(\frac{1}{100}\Big)^{k+1}\Big(1-O\Big(\frac{\eta m \log(C(\pi)T)}{2^{k}}\Big)\Big) \ge \frac 1{2 (100)^{k+1}}.  
\end{align*}
On the other hand, by \cref{lem4.3}, we obtain (for all $k \in [K,2K]$) 
\[
\Big|\frac{\eta^{k+1}}{k!}\Big(\frac{L'}{L}(s,\pi)\Big)^{(k)}\Big| \leq \eta^2\int_{N_0}^{N_1}\Big|\sum_{N_0\leq n\leq u}\frac{\lambda_{\pi}(n)\Lambda(n)}{n^{1+i\tau}}\Big|\frac{du}{u} + 
\frac{1}{4(100)^{k+1}}, 
\]
where we bounded the error term $O( (110)^{-k} (m\eta \log (C(\pi) T)))$ in \cref{lem4.3} by $\frac 14 (100)^{-k-1}$.  Combining these two estimates, we conclude that if there is a zero $\rho_0$ with $|1+i\tau- \rho_0| \le \eta$ then 
\begin{equation*}
\label{eqn:zero_detector} 
 1\le  4(100)^{2K+1} 
 \eta^2 \int_{N_0}^{N_1}\Big|\sum_{N_0\leq n\leq u}\frac{\lambda_{\pi}(n)\Lambda(n)}{n^{1+i\tau}}\Big|\frac{du}{u}.
\end{equation*}
Squaring the above estimate and using Cauchy-Schwarz,  we obtain 
\begin{align*} 
1 &\ll (100)^{4K} \eta^4 \Big( \int_{N_0}^{N_1} \frac{du}{u} \Big) \Big( \int_{N_0}^{N_1} \Big| \sum_{N_0 \le n \le u} \frac{\lambda_{\pi}(n) \Lambda(n)}{n^{1+i\tau}} \Big|^2 \frac{du}{u} \Big) \\ 
&\ll (101)^{4K} \eta^3 \int_{N_0}^{N_1} \Big| \sum_{N_0 \le n \le u} \frac{\lambda_{\pi}(n) \Lambda(n)}{n^{1+i\tau}} \Big|^2 \frac{du}{u}, 
\end{align*}
since $\log (N_1/N_0) \ll K/\eta$.   
Since there are $\ll m \eta \log (C(\pi) T)$ zeros satisfying $|1+i\tau - \rho| \le \tau$, we may also recast the above estimate as (for $200\eta \le |\tau |\le T$) 
\begin{equation*}
\frac{\# \{ \rho = \beta+i\gamma: \ \beta \ge 1-\eta/2, \ |\gamma-\tau| \le \eta/2 \} }{m\eta \log (C(\pi) T)}  \ll 101^{4K} \eta^3  \int_{N_0}^{N_1}\Big|\sum_{N_0\leq n\leq u}\frac{\lambda_{\pi}(n)\Lambda(n)}{n^{1+i\tau}}\Big|^2 \frac{du}{u}.
\end{equation*}
Integrating both sides above over $200\eta \le |\tau | \le T$ we conclude that 
\begin{align} 
\label{4.5} 
 \# \{ \rho = \beta+i\gamma: \ \beta \ge 1-\eta/2, &\ 200\eta \le |\gamma|\le T\} \nonumber \\ 
 & \ll 101^{4K} \eta^3  m\log (C(\pi) T)  \int_{-T}^{T}  \int_{N_0}^{N_1}\Big|\sum_{N_0\leq n\leq u}\frac{\lambda_{\pi}(n)\Lambda(n)}{n^{1+i\tau}}\Big|^2 \frac{du}{u} d\tau.
\end{align}

We now work on bounding the right side of \eqref{4.5}, which is clearly 
\begin{align}
\label{4.6}  
&\ll 101^{4K} \eta^3m \log (C(\pi) T) \log (N_1/N_0)\max_{u \in [N_0, N_1]} \Big(  \int_{-T}^{T} \Big| \sum_{N_0\leq n\leq u}\frac{\lambda_{\pi}(n)\Lambda(n)}{n^{1+i\tau}}\Big|^2 d\tau \Big)\nonumber \\ 
&\ll 102^{4K} \eta^2 m \log (C(\pi) T)  \max_{u \in [N_0, N_1]} \Big(  \int_{-T}^{T} \Big| \sum_{N_0\leq n\leq u}\frac{\lambda_{\pi}(n)\Lambda(n)}{n^{1+i\tau}}\Big|^2 d\tau \Big).
\end{align}
We bound the integral in the above display by an application of Plancherel, as in Gallagher \cite[Theorem 1]{Gallagher2}:  for $T\geq1$ and any sequence of complex numbers $\{a_n\}_{n=1}^{\infty}$ one has 
\[
\int_{-T}^{T}\Big|\sum_{n=1}^{\infty} a_n n^{-it}\Big|^2 dt\ll T^2\int_0^{\infty}\Big|\sum_{n\in(w,we^{1/T}]}a_n\Big|^2\frac{dw}{w}.
\]
Applying Gallagher's bound, we deduce that for any $u \in [N_0,N_1]$ 
\begin{align*}
\int_{-T}^{T}\Big|\sum_{N_0\leq n\leq u}\frac{\lambda_{\pi}(n)\Lambda(n)}{n^{1+i\tau}}\Big|^2 d\tau &\ll T^2\int_{0}^{\infty}\Big|\sum_{\substack{x<n\leq x e^{1/T} \\ N_0\leq n\leq N_2}}\frac{\lambda_{\pi}(n)\Lambda(n)}{n}\Big|^2\frac{dx}{x}\\
&\ll T^2\int_{N_0/e}^{N_1}\Big|\sum_{x<n\leq x e^{1/T}}|\lambda_{\pi}(n)|\Lambda(n)\Big|^2\frac{dx}{x^3}.
\end{align*}
Appealing now to \eqref{1.10} (which applies because of \eqref{eqn:K_def}), we find that the above is 
\[
\ll_{m} T^2\int_{N_0/e}^{N_1}\frac{x^2}{T^2}\frac{dx}{x^3}\ll_{m}  \frac{K}{\eta}.  
\]
Using this in \eqref{4.6}, we conclude that this quantity is bounded by 
\[
\ll 102^{4K} K \eta m \log (C(\pi) T) \ll 105^{4K}. 
\]

Inserting the above bound in \eqref{4.5}, and noting that there are $\ll \eta m \log (C(\pi)T) \ll K$ zeros with $\beta >1-\eta/2$ and $|\gamma|\le 200\tau$, we obtain 
\[
N_{\pi}(1-\eta/2, T) \ll 105^{4K}.
\]
This estimate implies our theorem in the range $1/\log (C(\pi)T) \le 1-\sigma \le 1/(400m)$.  In the range $1-\sigma \le 1/\log (C(\pi) T)$, 
simply bound $N_{\pi} (\sigma, T)$ by $N_{\pi}(1-1/\log (C(\pi)T), T)$.  In the range $1-\sigma >1/(400m)$, the theorem is trivial 
since there are $\ll mT \log (C(\pi)T)$ zeros with $\beta \in (0,1)$ and $|\gamma| \le T$.    
\end{proof}

\section{Proof of \cref{thm:bound_1/2,thm:weak_subconvexity}}

\noindent Let $L(s,\pi)\in\mathcal{S}(m)$, and in proving the theorem we may plainly suppose that $L(1/2,\pi) \neq 0$.   Our starting point is Heath-Brown's argument to establish a sharp convexity bound for $L$-functions.  
This begins with a variant of Jensen's formula, connecting $\log |L(\frac 12,\pi)|$ with zeros lying in the critical strip $0<\Re(s)<1$.
 The Jensen formula that we need is 
 \begin{align} 
 \label{5.1} 
\log|(1/2)^r L(1/2,\pi )| &+\sum_{\substack{\rho=\beta+i\gamma \\ 0<\beta<1}}\log\Big|\cot\Big(\frac{\pi}{2}\Big(\rho-\frac{1}{2}\Big)\Big)\Big|+\sum_{\re(\mu_{\pi}(j))<0}\log\Big|\cot\Big(\frac{\pi}{2}\Big(\mu_{\pi}(j)+\frac{1}{2}\Big)\Big)\Big| \nonumber\\
&=\frac{1}{2}\int_{-\infty}^{\infty}\log|L(1+it,\pi )L(it,\pi )t^r(1-it)^{r}|\frac{dt}{\cosh(\pi t)}.
\end{align}
This may be established as in Heath-Brown \cite{HB}, or applying \cite[Lemma 3.1, p. 207]{BS} with $F(s) = (s-1)^r L(s,\pi)$ and $x=1/2$.  
The proof is by conformally mapping the strip $z=x+iy$ with $0< x< 1$ onto the unit disc $|\zeta| <1$ by means of 
the substitution $\zeta = (e^{\pi iz} -i)/(e^{\pi i z} + i)$, and then using the usual Jensen formula for the unit disc.

Now if $z= x+iy$ is a complex number with $|x| \le \frac 12$, then a small calculation gives 
\begin{equation} 
\label{5.2}  
\log |\cot (\pi z/2)| = \frac 12 \log \Big( \frac{\cosh (\pi y) + \cos (\pi x)}{\cosh (\pi y) -\cos (\pi x)} \Big) 
\ge \frac{\cos (\pi x)}{\cosh (\pi y)}, 
\end{equation}
where the last inequality follows because $\frac{1}{2}\log ((1+t)/(1-t)) \ge t$ for $1> t \ge 0$ by a Taylor expansion.  From \eqref{5.2} and since Re$(\mu_\pi (j)) > -1$, the terms $\log |\cot (\pi (\mu_{\pi}(j) +1/2)/2)|$ appearing in \eqref{5.1} are all non-negative.   Bounding the sum over zeros below using \eqref{5.2}, we conclude that the left side of \eqref{5.1} is at least 
\begin{equation} 
\label{5.3} 
\log |L(1/2, \pi)|  + \sum_{\substack{\rho=\beta+i\gamma \\ 0<\beta<1}} \frac{\sin (\pi\beta)}{\cosh (\pi \gamma)}. 
\end{equation} 

Now we consider the right side of \eqref{5.1}.  Using the functional equation to connect $L(it,\pi)$ with $L(1-it, \widetilde{\pi})$, and then using Stirling's formula, we obtain 
\begin{align*}
\log |L(it,\pi)| &= \log |L(1-it,\widetilde{\pi})| + \frac{1}{2} \log N_\pi + \sum_{j=1}^{m} \log \Big| \frac{\Gamma((1+\mu_{\widetilde{\pi}}(j)-it)/2)}{\Gamma((\mu_\pi(j) +it)/2)}\Big| + O(m) \\ 
&= \log |L(1+it, \pi)| + \frac{1}{2} \log N_{\pi} + \frac 12 \sum_{j=1}^{m} \log (1+ |\mu_{\pi}(j) + it|)  + O(m^2)\\ 
&\le \log |L(1+it, \pi)| + \frac 12 \log C(\pi) + \frac{m}{2} \log(1+|t|)  + O(m^2).
\end{align*}
Thus the right side of \eqref{5.1} is bounded by 
\begin{align}
\label{5.4}
&\frac{1}{4}\log C(\pi) +\int_{-\infty}^{\infty} \Big( \log|t^r L(1+it,\pi )| + \frac{m}{4} \log (1+|t|) + O(m^2)\Big) \frac{dt}{\cosh(\pi t)} \nonumber\\
=~&\frac 14\log C(\pi) + O(m^2)  + \int_{-\infty}^{\infty} \log |t^rL(1+it, \pi)| \frac{dt}{\cosh (\pi t)}. 
 \end{align}
Since $|t^r L(1+it,\pi)|$ grows at most polynomially in $|t|$, and $1/\cosh (\pi t)$ decreases exponentially in $|t|$, we 
may see that 
\begin{align*} 
\int_{-\infty}^{\infty}\log|t^r L(1+it,\pi )|\frac{dt}{\cosh(\pi t)}&=\lim_{\eta\to 0^+}\Re\Big(\int_{-\infty}^{\infty}\log (t^r L(1+\eta+it,\pi ))\frac{dt}{\cosh(\pi t)}\Big)\\
&=\lim_{\eta\to 0^+}\Re\Big(\sum_{n=2}^{\infty}\frac{\lambda_{\pi }(n)\Lambda(n)}{n^{1+\eta}\log n}\int_{-\infty}^{\infty}n^{-it}\frac{dt}{\cosh(\pi t)}\Big)+O(m).\\
\end{align*}
Now 
\[
\int_{-\infty}^{\infty} n^{-it}\frac{dt}{\cosh (\pi t)} = \frac{1}{\cosh ((\log n)/2)}= \frac{2}{\sqrt{n} +1/\sqrt{n}} = \frac{2}{\sqrt{n}} + O\Big(\frac{1}{n^{\frac 32}}\Big),
\]
and therefore 
\begin{align*}
\int_{-\infty}^{\infty}\log|t^r L(1+it,\pi )|\frac{dt}{\cosh(\pi t)} &=2\Re\Big(\sum_{n=2}^{\infty}\frac{\lambda_{\pi }(n)\Lambda(n)}{n^{3/2}\log n}\Big)+O\Big(\sum_{n=2}^{\infty}\frac{|\lambda_{\pi }(n)|\Lambda(n)}{n^{5/2}\log n}+m\Big) \\
&=2\log|L(3/2,\pi )|+O(m).
\end{align*}

Combining the above remarks with \eqref{5.3} and \eqref{5.4}, we conclude that 
\begin{equation} 
\label{5.6} 
\log|L(1/2,\pi )|\leq\frac{1}{4}\log C(\pi )-\sum_{\substack{\rho=\beta+i\gamma\\0<\beta<1}}\frac{\sin (\pi\beta)}{\cosh (\pi \gamma)} +2\log|L(3/2,\pi)|+O(m^2).
\end{equation}
All this follows closely the work of Heath-Brown, except that we have kept a negative contribution from the zeros of $L(s,\pi)$ which 
we shall now bound from below.  

\begin{proof}[Proof of \cref{thm:bound_1/2}]
Plainly for any positive real number $T$, and any $\frac 12 \ge \delta >0$  we have 
\begin{equation*} 
\sum_{\substack{\rho=\beta+i\gamma\\0<\beta<1}}\frac{\sin (\pi\beta)}{\cosh (\pi \gamma)}  \ge 
 \sum_{ \substack{ \rho = \beta+i\gamma \\ |\gamma| \le T} } \frac{\sin (\pi \beta)}{\cosh (\pi T)} 
\ge \frac{\sin (\pi \delta)}{\cosh (\pi T)} \sum_{\substack{\rho = \beta +i\gamma \\ \delta \le \beta \le 1-\delta \\ 
|\gamma |\le T }} 1\ge \frac{2\delta}{\cosh (\pi T)} \sum_{\substack{\rho = \beta +i\gamma \\ \delta \le \beta \le 1-\delta \\ 
|\gamma |\le T }} 1.
\end{equation*} 
The functional equation combined with complex conjugation shows that if $\beta+i\gamma$ is a zero then 
so is $1-\beta +i\gamma$.  Thus, choosing $T=6$ and invoking \eqref{3.4}, we obtain  
\begin{equation*}
\sum_{\substack{\rho = \beta +i\gamma \\ \delta \le \beta \le 1-\delta \\ 
|\gamma |\le 6 }} 1=N_{\pi}(0,6)-2N_{\pi}(1-\delta,6) 
\ge \frac{4}{15} \log C(\pi) - 2 N_{\pi}(1-\delta, 6)+O(m).
\end{equation*}
Therefore 
\[
\sum_{\substack{\rho=\beta+i\gamma\\0<\beta<1}}\frac{\sin (\pi\beta)}{\cosh (\pi \gamma)}\geq\frac{2\delta}{\cosh (6\pi)}\Big(\frac{4}{15}\log C(\pi)-2N_{\pi}(1-\delta,6)\Big)+O(m).
\]
Inserting this lower bound into \eqref{5.6}, we obtain \cref{thm:bound_1/2}.
\end{proof}

\begin{proof}[Proof of \cref{thm:weak_subconvexity}]  Choose $\delta = 10^{-8} m^{-3} (\log \log C(\pi))/\log C(\pi)$.  Then Theorem \ref{thm:LFZDE} 
gives $N_{\pi}(1-\delta, 6) \ll_m \sqrt{\log C(\pi)}$.  	Inserting this bound in Theorem \ref{thm:bound_1/2}, the corollary follows. 
 \end{proof}

\section{Proof of \cref{thm:weaker_ramanujan_1}}

\noindent We fix a nonnegative smooth function $\Phi$ supported in $(-2,2)$, say, and write 
\begin{equation} 
\label{7.1} 
{\check \Phi}(s) = \int_{-\infty}^{\infty} \Phi(y) e^{sy} dy. 
\end{equation} 
Thus ${\check \Phi}(s)$ is an entire function of $s$, and by integrating by parts many times we obtain for any integer $k\ge 0$ 
\begin{equation} 
\label{7.2} 
|{\check \Phi}(s)| \ll_{\Phi, k}  \frac{e^{2 |\Re(s)|}}{|s|^k}. 
\end{equation} 
Let $T \ge 1$ be a real parameter, and note that by Mellin (or Fourier) inversion one has (for any positive real number $x$, and any real $c$) 
\begin{equation} 
\label{7.3} 
T\Phi(T \log x) = \frac{1}{2\pi i} \int_{c-i\infty}^{c+i\infty} {\check \Phi}(s/T) x^{-s} ds. 
\end{equation}  

Recall that 
\[
L(s,\pi\times\widetilde{\pi})=\sum_{n\geq 1}\frac{a_{\pi\times\widetilde{\pi}}(n)}{n^s}=\prod_p L_p(s,\pi\times\widetilde{\pi}),
\]
with 
\begin{equation}
\label{7.4}
L_{p}(s,\pi\times\widetilde{\pi})=\prod_{j_1=1}^{m}\prod_{j_2=1}^{m}\Big(1-\frac{\alpha_{j_1,j_2,\pi\times\widetilde{\pi}}(p)}{p^s}\Big)^{-1}=1+\sum_{j=1}^{\infty}\frac{a_{\pi\times\widetilde{\pi}}(p^j)}{p^{js}}.  
\end{equation}
The Rankin-Selberg $L$-function $L(s,\pi \times \widetilde \pi)$ has non-negative coefficients, converges in Re$(s)>1$, 
and extends to the complex plane with a simple pole at $s=1$.  

Our proof of  the Brun-Titchmarsh result Theorem \ref{thm:weaker_ramanujan_1} will be based on an application of the Selberg sieve.  To 
pave the way for this, given a square-free number $d$ we need  an asymptotic formula for 
\[
\sum_{d|n} a_{\pi \times \widetilde \pi} (n) \Phi \Big( T \log \frac nx\Big), 
\]
which we establish in the following lemma.

\begin{lem} \label{lem7.2} Let $\pi \in {\mathcal A}(m)$, and $\Phi$ be as above.   Let $d\ge 1$ be a square-free integer.  For any $x > 0$ and $T\ge 1$ we have 
\[
\sum_{d|n} a_{\pi \times {\widetilde \pi}}(n) \Phi\Big(T \log \frac nx \Big)   =  \kappa g(d) \frac{x}{T} {\check \Phi}(1/T) +  O_m \Big ( x^{\frac 12}C(\pi \times \widetilde \pi) d^{m^2} T^{m^2} \Big), 
\]
where 
\[
\kappa = \mathop{\textup{Res}}_{s=1} L(s,\pi \times \widetilde \pi), \qquad \text{and} \qquad g(d) = \prod_{p|d} (1- L_p(1,\pi \times \widetilde \pi)^{-1}). 
\]  
\end{lem}

\begin{proof}   Using \eqref{7.3}, we may write (for any real number $c>1$) 
\[
\sum_{d|n} a_{\pi \times {\widetilde \pi}}(n) \Phi\Big(T \log \frac nx \Big)   = \frac{1}{2\pi i T} \int_{c-i\infty}^{c+i\infty} {\check \Phi}(s/T) x^s \sum_{d|n} 
\frac{a_{\pi \times\widetilde\pi}(n)}{n^s} ds. 
\]
The Dirichlet series appearing above has non-negative coefficients and converges in the region Re$(s) >1$, and 
matches the Rankin-Selberg $L$-function $L(s, \pi \times \widetilde \pi)$ except for the Euler factors at primes $p$ dividing $d$.  Indeed, by multiplicativity, we 
may write 
\[
\sum_{d|n} \frac{a_{\pi \times\widetilde\pi}(n)}{n^s} = \prod_{p\nmid d} L_p(s,\pi \times {\widetilde \pi}) \prod_{p|d} \Big( \sum_{j=1}^{\infty} \frac{a_{\pi \times \widetilde \pi}(p^j)}{p^{js}} \Big) 
= L(s, \pi \times {\widetilde \pi}) g_d(s, \pi \times {\widetilde \pi}), 
\]
where 
\begin{equation} 
\label{7.5} 
g_d(s, \pi \times {\widetilde \pi}) = \prod_{p|d} (1- L_p(s, \pi \times \widetilde \pi)^{-1}) = 
\prod_{p|d} \Big(1 - \prod_{j_1, j_2=1}^{m} \Big(1 - \frac{\alpha_{j_1,j_2, \pi \times{\widetilde \pi} } (p)}{p^s} \Big)\Big). 
\end{equation} 
Thus the integral above equals 
\begin{equation} 
\label{7.6} 
\frac{1}{2\pi i T} \int_{c-i\infty}^{c+i\infty} {\check \Phi}(s/T) { L(s,\pi \times \widetilde \pi) } g_d(s,\pi \times \widetilde \pi) x^s ds. 
\end{equation}

We evaluate \eqref{7.6} by moving the line of integration to Re$(s)=1/2$.  We encounter a simple pole at $s=1$ and the residue 
here is the main term appearing in our lemma; note that $g(d) = g_d(1,\pi\times \widetilde \pi)$.   To bound the integral on the line Re$(s)=1/2$, using the Phragm{\'e}n-Lindel{\"o}f principle and 
Stirling's formula (using that on the line Re$(s) =5/2$ we have $L(s,\pi \times \widetilde \pi) \ll 1$) we find 
\[
|L(\tfrac 12+it , \pi \times \widetilde \pi)| \ll C(\pi \times \widetilde \pi) (2 + |t|)^{m^2}. 
\]
Further, since $|\alpha_{j_1,j_2,\pi \times \widetilde \pi}(p)| \le p$ for all $j_1$, $j_2$ and $p$, from the definition \eqref{7.5} it 
follows that   
\[
|g_d(\tfrac 12+it,\pi \times \widetilde\pi)| \le \prod_{p|d}\Big( 1 +   \Big(1 + p^{\frac 12} \Big)^{m^2} \Big) \le (2d)^{m^2}. 
\]
Thus the integral on the Re$(s)=1/2$ line is 
\[
\ll \frac{\sqrt{x}}{T} C(\pi \times \widetilde \pi) (2d)^{m^2} \int_{-\infty}^{\infty} (2+|t|)^{m^2} \Big|{\check \Phi}\Big( \frac 1T \Big(\frac 12 +it \Big) \Big) \Big| dt. 
\]
Using \eqref{7.2} with $k=0$ for $|t| \le T$, and $k= m^2 + 2$ for $|t| >T$, we  see that the above is 
\[
 \ll_{m, \Phi}  \frac{\sqrt{x}}{T} C(\pi \times \widetilde \pi) d^{m^2} 
\int_{-\infty}^{\infty} (2+|t|)^{m^2} \min \Big( 1, \frac{T^{m^2+2}}{(2+|t|)^{m^2+2}}\Big) dt \ll_{m, \Phi} \sqrt{x}  C(\pi \times \widetilde \pi) d^{m^2} T^{m^2}.  
\]
\end{proof}  
 
 From Lemma \ref{lem7.2} and an application of the Selberg sieve we shall obtain the following proposition. 
 
 \begin{prop} 
 \label{prop7.3}  Keep the notations of Lemma \ref{lem7.2}.  Then for any $x>0$, $T\ge 1$, and $z\gg_{m} C(\pi\times\widetilde{\pi})^4$, we have 
\[
\sum_{\substack{ n \\ p|n \implies p>z}} a_{\pi \times \widetilde \pi}(n) \Phi\Big( T \log \frac nx \Big) \le  \frac{3x}{ T\log z}{\check \Phi}(1/T)+O_{m,\Phi}(x^{\frac 12} C(\pi \times \widetilde \pi) T^{m^2} z^{2m^2 +3}). 
\]
 \end{prop} 
\begin{proof}  As mentioned already, this follows from a standard application of Selberg's sieve and Lemma \ref{lem7.2}, 
see for example Theorem 7.1 of \cite{FI}.   Using Theorem 7.1 of \cite{FI} and \eqref{7.3} there (with $D=z^2$ in their 
notation), we find 
\begin{align} 
\label{7.7}
\sum_{\substack{ n \\ p|n \implies p>z}} a_{\pi \times \widetilde \pi}(n) \Phi\Big( T \log \frac nx \Big)& \le \kappa \frac{x}{T} 
{\check \Phi}(1/T)  \Big( \sum_{\substack{d | \prod_{p\le z}p \\ d\le z} } \prod_{p|d} \frac{g(p)}{1-g(p)} \Big)^{-1} \nonumber \\ 
&\hskip 1 in + O_{m,\Phi}
\Big( x^{\frac 12} C(\pi \times \widetilde \pi) T^{m^2} \sum_{d\le z^2} d^{m^2} \tau_3(d)\Big). 
\end{align} 
Here $\tau_3(d)$ is the number of ways of writing $d$ as a product of three natural numbers.  Since $\tau_3(d) \ll_{\epsilon}  d^{\epsilon}$ 
we may trivially bound the error term in \eqref{7.7} by $\ll_{m,\Phi} x^{\frac 12} C(\pi \times \widetilde \pi) T^{m^2} z^{2m^2 +3}$.

For the first sum, we observe from the definitions of $g(p)$ and $L_p(s,\pi\times\widetilde{\pi})$ that
\begin{equation}
\label{7.8} 
\sum_{\substack{d\mid \prod_{p\le z} p \\ d\leq z}}\prod_{p\mid d}\frac{g(p)}{1-g(p)}\ge \sum_{\substack{n\leq z \\ n\text{ square-free}}}\prod_{p\mid n}\sum_{j=1}^{\infty}\frac{a_{\pi\times\widetilde{\pi}}(p^j)}{p^j}\geq \sum_{\substack{n\leq z}}\frac{a_{\pi\times\widetilde{\pi}}(n)}{n}.
\end{equation} 
Let $\Phi_1$ be a non-negative smooth function supported on $[0,1]$, with $\Phi_1(t) =1$ for $\epsilon \le t \le 1-\epsilon$ and $\Phi_1(t) \le 1$ for $0\le t\le 1$.  
Then appealing to Lemma \ref{lem7.2} with $d=1$ and $T=1$ there we obtain that 
\[
\sum_{y\le n \le ey} \frac{a_{\pi\times \widetilde{\pi}}(n)}{n} \ge \frac{1}{ey} \sum_{n} a_{\pi \times \widetilde{\pi}}(n) \Phi_1\Big( \log \frac{n}{y}\Big) = \frac{1}{e} (e-1+O(\epsilon)) \kappa + O_m\Big( y^{-\frac 12} C(\pi \times \widetilde{\pi})\Big).
\]
Dividing the interval $[\sqrt{z},z]$ into blocks of the form $[y,ey]$, it follows that 
\[
\sum_{\sqrt{z} \le n \le z} \frac{a_{\pi \times \widetilde{\pi}}(n)}{n} \ge \frac{\kappa}{3} \log z + O_m\Big(z^{-\frac 14} C(\pi \times \widetilde{\pi})\Big). 
\]
 Therefore, if $z\gg_{m} C(\pi\times\tilde{\pi})^4$, then
\begin{align*}
	\sum_{n\leq z}\frac{a_{\pi\times\widetilde{\pi}}(n)}{n}\geq 1+\sum_{\sqrt{z}<n\leq z}\frac{a_{\pi\times\widetilde{\pi}}(n)}{n}\geq \frac{1}{3} (1+ {\kappa}\log z).
\end{align*}
Using this bound in \eqref{7.8} and then in \eqref{7.7}, and noting that for all $\kappa >0$ one has $\kappa/(1+\kappa \log z) \le 1/\log z$, 
the proposition follows.  
 \end{proof}

\begin{proof}[Proof of \cref{thm:weaker_ramanujan_1}]   Since Theorem \ref{thm:weaker_ramanujan_1} follows from \eqref{eqn:BT} for $m=1$, we may assume below that $m\ge 2$.   
Suppose that $x\gg_m C(\pi \times \widetilde{\pi})^{36m^2}$, and that $ 1\le T \le x^{\frac{1}{9m^2}}$.  Take $z= x^{\frac{1}{9m^2}}$, 
and $\Phi$ to be a smooth non-negative function supported in $(-\epsilon, 1+\epsilon)$ with $\Phi(t)=1$ for $0\le t\le 1$.  An application of Proposition \ref{prop7.3} gives 
\[
\sum_{\substack{ x<n\leq x e^{1/T} \\ p|n \implies p>z}} a_{\pi \times \widetilde \pi}(n) \ll_m \frac{x}{T\log z}+x^{\frac 12} C(\pi \times \widetilde \pi) T^{m^2} z^{2m^2 +3} 
\ll_m \frac{x}{T\log x}.
\]
The left hand side above includes all prime powers $p^k$ in $(x,xe^{1/T}]$ with $p>z$, and so we conclude that 
\begin{equation} 
\label{7.9} 
\sum_{\substack{ x < p^k \le xe^{1/T} \\ k \le 9m^2 }} a_{\pi \times \widetilde{\pi}}(p^k) \ll_m \frac{x}{T \log x}.
\end{equation} 

In Theorem \ref{thm:weaker_ramanujan_1}, we are interested in bounding $\lambda_{\pi \times \widetilde{\pi}}(p^k)$ in place of $a_{\pi \times \widetilde{\pi}}(p^k)$ above.  
Note that from \eqref{1.2} that for any given prime $p$, we have the formal identity
	 \[
	 \exp\Big(\sum_{k=1}^{\infty}\frac{\lambda_{\pi\times\widetilde{\pi}}(p^k)}{k}X^k\Big)=1+\sum_{k=1}^{\infty}a_{\pi\times\widetilde{\pi}}(p^k)X^k. 
	 \]
Expanding both sides and comparing coefficients, from the non-negativity of $\lambda_{\pi\times\widetilde{\pi}}(p^k)$ and $a_{\pi\times\widetilde{\pi}}(p^k)$ 
we deduce that 
 \begin{equation}
	 \label{7.10}
	 a_{\pi\times\widetilde{\pi}}(p^k)\geq \frac{\lambda_{\pi\times\widetilde{\pi}}(p^k)}{k}.
	 \end{equation}

From \eqref{7.9} and \eqref{7.10} it follows that 
\[
\sum_{\substack{ x < n= p^k \le xe^{1/T} \\ k \le 9m^2 }} \lambda_{\pi \times \widetilde{\pi}}( n ) \Lambda(n) \ll_m \frac{x}{T}. 
\]
To complete the proof of Theorem \ref{thm:weaker_ramanujan_1} it remains lastly to bound the contribution of primes powers $p^k$ with $k > 9m^2$.  
Since there are very few such prime powers, it will be enough to use a crude bound on $\lambda_{\pi \times \widetilde{\pi}} (p^k)$.  From 
\eqref{eqn:LRS_auto_2} one obtains $\lambda_{\pi \times \widetilde{\pi}}(p^k) \le m^2 p^{k(1-1/m^2)}$, and so 
\[
\sum_{\substack{ x < n= p^k \le xe^{1/T} \\ k > 9m^2 }} \lambda_{\pi \times \widetilde{\pi}}( n ) \Lambda(n) \ll_m x^{1-\frac{1}{m^2}} 
\sum_{\substack{ p^k \le ex \\ k > 9m^2} } \Lambda(m) \ll_m x^{1-\frac{8}{9m^2}} \ll_m \frac{x}{T}. 
\]
This finishes the proof of Theorem \ref{thm:weaker_ramanujan_1}. 
 \end{proof}

\appendix

\section{An inequality on Rankin-Selberg coefficients,}
\label{sec:appendix}

\begin{center} {\sc by Farrell Brumley}\footnote{LAGA - Institut Galil\'ee, 99 avenue Jean Baptiste Cl\'ement, 93430 Villetaneuse, France, \url{brumley@math.univ-paris13.fr}}\footnote{Supported by ANR grant 14-CE25}
\end{center}

\bigskip

Let $\pi,\pi'$ be irreducible unitary generic representations of $\GL_m(\Q_p)$ and $\GL_{m'}(\Q_p)$, respectively. Let $L(s,\pi\times\pi')$ be the local Rankin-Selberg $L$-factor. Write its logarithm as 
\[
\log L(s,\pi\times\pi')=\sum_{f\geqslant 1}\frac{\lambda_{\pi\times\pi'}(p^f)}{fp^{fs}}.
\]
Our aim is to prove the following inequality.\footnote{I would like thank Kannan Soundararajan and Jesse Thorner for allowing me to include this appendix to their paper, and for helpful discussions regarding the proof during a visit to Stanford.}

\begin{prop}\label{main} For every $f\geqslant 1$ we have
\[
|\lambda_{\pi\times\pi'}(p^f)|\leq \sqrt{\lambda_{\pi \times \tilde{\pi}} (p^f) \lambda_{\pi' \times \tilde{\pi}'} (p^f)}  \leq \frac12 \big(\lambda_{\pi\times\tilde\pi}(p^f)+\lambda_{\pi'\times\tilde\pi'}(p^f)\big).
\]
\end{prop}

The model computation is when $\pi$ and $\pi'$ are both unramified. In this case, the proposition is immediate from the well-known expression for the local Rankin-Selberg $L$-factor
\begin{equation}\label{unramified}
L(s,\pi\times\pi')=\prod_{j=1}^m\prod_{k=1}^{m'}(1-\alpha_\pi(p,j)\alpha_{\pi'}(p,k)p^{-s})^{-1}
\end{equation}
in terms of the Satake parameters $\alpha_\pi(p,j)$ and $\alpha_{\pi'}(p,k)$. From this it follows that the coefficients $\lambda_{\pi\times\pi'}(p^f)$ are given by
\begin{equation}\label{unramified-coeff}
\lambda_{\pi\times\pi'}(p^f)=\sum_{j=1}^m\sum_{k=1}^{m'}\alpha_\pi(p,j)^f\alpha_{\pi'}(p,k)^f = \lambda_\pi(p^{f}) \lambda_{\pi'}(p^f).
\end{equation}
Similarly, in the unramified situation, $\lambda_{\pi \times \tilde \pi}(p^f) = |\lambda_{\pi} (p^f)|^2$ and $\lambda_{\pi' \times \tilde{\pi}'}(p^f) = 
|\lambda_{\pi'}(p^f)|^2$.   Thus $|\lambda_{\pi \times \pi'}(p^f)| = \sqrt{\lambda_{\pi \times \tilde{\pi}}(p^f) \lambda_{\pi' \times \tilde \pi'}(p^f)}$ 
and the proposition follows from the inequality $|AB|\leq \frac{|A|^2+|B|^2}{2}$ of geometric and arithmetic means. 
The proof of Proposition \ref{main} follows along the same lines, but we shall need a more explicit description of the Rankin-Selberg local $L$-factors. The main issue is that, contrary to the unramified case, the local roots of the Rankin-Selberg convolution are not simply the products of the local roots of the standard $L$-function.

\subsection{Description of local Rankin-Selberg factor}
In this section we describe the local Rankin-Selberg $L$-function $L(s,\pi\times\pi')$ in terms of representation theoretic data. The main identity is display \eqref{expression} below. We follow closely the exposition in \cite[Appendix A]{RS}, where the case when $\pi'\simeq\tilde\pi$ was explicated.

We begin by realizing $\pi$ as a Langlands quotient
\begin{equation}\label{Langlands-quotient}
\pi=J(G,P;\tau_1[\sigma_1],\ldots ,\tau_r[\sigma_r]).
\end{equation}
Here $G=\GL_m(\Q_p)$, $P$ is a standard parabolic of $G$ corresponding to the partition $(m_1,\ldots ,m_r)$ of $m$, $\tau_j$ is a tempered representation of $\GL_{m_j}(\Q_p)$, the real numbers $\sigma_j$ satisfy $\sigma_1\geqslant \cdots \geqslant \sigma_r\geqslant 0$, and $\tau[\sigma]$ denotes the representation $\tau\otimes |\det|^\sigma$. Similar notation holds for $\pi'$. Then
\begin{equation}\label{red-2-temp}
L(s,\pi\times\pi')=\prod_{j=1}^r\prod_{k=1}^{r'}L(s+\sigma_j+\sigma_k',\tau_j\times\tau'_k).
\end{equation}

Next we use the fact that tempered representations of $\GL_m(\Q_p)$ are fully induced representations from discrete series. Moreover, discrete series themselves can be constructed as generalized Speh representations, obtained through an induction procedure from supercuspidals as follows. For any discrete series representation $\delta$ on $\GL_m(\Q_p)$ there is a divisor $d\mid m$ and a unitary supercuspidal representation $\rho$ on $\GL_d(\Q_p)$ such that $\delta$ is isomorphic to the unique square-integrable subquotient of the representation 
\[
\bigtimes_{\nu=1}^n \rho[\nu-(n+1)/2]
\]
induced from the standard Levi
\[
\underbrace{\GL_d(\Q_p)\times\cdots \times\GL_d(\Q_p)}_n,
\]
where $n=m/d$.

We apply this for every $\tau_j$ appearing in \eqref{Langlands-quotient}, to obtain integers $d_j\mid m_j$, $n_j=m_j/d_j$, and unitary supercuspidals $\rho_j$ on $\GL_{d_j}(\Q_p)$. We proceed similarly for $\pi'$. Using induction by stages (to combine the reduction of tempered representations $\tau$ to discrete series $\delta$ with the reduction of discrete series $\delta$ to supercuspidals $\rho$) we obtain
\begin{equation}\label{red-2-rho}
L(s,\pi\times\pi')=\prod_{j=1}^r\prod_{k=1}^{r'}\prod_{\nu=1}^{\min (n_j,n_k')} L\left(s+\sigma_j+\sigma_k'+\frac{n_j+n_k'}{2}-\nu,\rho_j\times\rho'_k\right).
\end{equation}

We now organize the $\rho_j$ and $\rho_k'$ into twist-equivalence classes. Let
\begin{enumerate}
\item $\underline{J}=[J_1,\ldots ,J_A]$ be a set partition of $\{1,\ldots ,r\}$;
\item $\underline{K}=[K_1,\ldots ,K_B]$ be a set partition of $\{1,\ldots ,r'\}$;
\item $\{\varrho_1,\ldots , \varrho_L\}$ be a set of unitary twist-inequivalent supercuspidal representations $\varrho_\ell$ of a general linear group over $\Q_p$,
\end{enumerate}
with the property that 
\begin{enumerate}
\smallskip
\item for every $a\in \{ 1,\ldots ,A\}$ there is a $\ell=\ell(a)\in \{ 1,\ldots ,L\}$ and for every $j\in J_a$ there is $t_j\in\R$ such that $\rho_j\simeq \varrho_{\ell}[it_j]$;
\smallskip
\item for every $b\in \{ 1,\ldots ,B\}$ there is a $\ell'=\ell'(b)\in \{ 1,\ldots ,L\}$ and for every $k\in K_b$ there is $t_k'\in\R$ such that $\rho_k'\simeq \varrho_{\ell'}[it_k']$;
\smallskip
\item the assignments $a\mapsto \ell(a)$ and $b\mapsto \ell'(b)$ are injective.
\end{enumerate}
In this way, for any $a\in\{ 1,\ldots ,A\}$, the set $\{\rho_j: j\in J_a\}$ consists of all those $\rho_j$ appearing in \eqref{red-2-rho} which are twist equivalent to some given $\varrho_{\ell(a)}$. We may assume, if we wish, that the set $\{\varrho_1,\ldots , \varrho_L\}$ is minimal for this property. Setting $s_j=\sigma_j+it_j$, $s_k'=\sigma_k'+it_k'$, and
\[
L_{J_a,K_b}(s)=\prod_{j\in J_a}\prod_{k\in K_b}\prod_{\nu=1}^{\min (n_j,n_k')} L\left(s+s_j+s_k'+\frac{n_j+n_k'}{2}-\nu,\varrho_{\ell(a)}\times\varrho_{\ell'(b)}\right),
\]
we obtain the following expression
\[
L(s,\pi\times\pi')=\prod_{a=1}^A\prod_{b=1}^BL_{J_a,K_b}(s).
\]

Now, many of the factors in the above product are simply $1$. Indeed, for supercuspidal representations $\varrho$ on $\GL_d(\Q_p)$ and $\varrho'$ on $\GL_{d'}(\Q_p)$, the local factor $L(s,\varrho\times\varrho')$ is 1 unless $\varrho$ is twist equivalent to $\varrho'$ (in which case $d=d'$). Otherwise, when $\varrho'=\varrho[\sigma]$, we have
\[
L(s,\varrho\times\varrho[\sigma])=L(s+\sigma,\varrho\times\varrho)=\big(1-p^{-e(\sigma+s)}\big)^{-1},
\]
where $e$ is the torsion number for $\varrho$. (The torsion number is the order of the finite cyclic group of characters $\chi=|\det |^u$ such that $\varrho\otimes\chi\simeq\varrho$.)

We deduce that
\[
L(s,\pi\times\pi')=\prod_{(a,b)\in \Delta} L_{J_a,K_b}(s),
\]
where
\[
\Delta=\big\{ (a,b)\in \{ 1,\ldots ,A\}\times \{ 1,\ldots ,B\}: \;\ell(a) = \ell'(b)\big\}.
\]
Let $\ell: \Delta\rightarrow \{ 1,\ldots ,L\}$ be the map sending $(a,b)$ to $\ell(a,b):=\ell(a)=\ell'(b)$; it is injective. If $e_\ell$ denotes the torsion number of $\varrho_\ell$, then
\[
L_{J_a,K_b}(s)=\prod_{j\in J_a}\prod_{k\in K_b}\prod_{\nu=1}^{\min (n_j, n_k')} \big(1-p^{-e_{\ell(a,b)}(s+s_j+s_k'+\frac{n_j+n_k'}{2}-\nu)}\big)^{-1}.
\]
Setting $z_j=p^{-s_j-n_j/2}$ and $z_k'=p^{-s_k'-n_k'/2}$, we obtain the formula
\begin{equation}\label{expression}
L(s,\pi\times\pi')=\prod_{(a,b)\in \Delta}\prod_{j\in J_a}\prod_{k\in K_b}\prod_{\nu=1}^{\min (n_j, n_k')} \big(1-(p^\nu z_j z_k')^{e_{\ell(a,b)}}p^{-e_{\ell(a,b)} s}\big)^{-1}.
\end{equation}

We now give some examples, to show that formula \eqref{expression} can be specialized to recover known cases.
\begin{example}\label{first-ex}
When $\pi'=\tilde\pi$, we have $r=r'$, $\underline{J}=\underline{K}$ (so that $A=B=L$), and the subset $\Delta$ is the diagonal copy of $\{ 1,\ldots ,A\}$ inside $\{ 1,\ldots ,A\}\times \{ 1,\ldots ,A\}$. Letting $\underline{F}=[F_1,\ldots ,F_L]$ denote the set partition $\underline{J}=\underline{K}$ of $\{1,\ldots ,r\}$, we recover in this case the formula
\begin{equation}\label{diagonal}
L(s,\pi\times\tilde\pi)=\prod_{l=1}^L\prod_{j,k\in F_l}\prod_{\nu=1}^{\min (n_j,n_k)}(1-(p^\nu z_j\bar{z}_k)^{e_l}p^{-e_ls})^{-1}
\end{equation}
of \cite[(A.12)]{RS}. 
\end{example}

\begin{example}
When $\pi$ and $\pi'$ are both principal series representations we have $r=m$, $r'=m'$, and  $n_j\equiv n_k'\equiv 1$. If, furthermore, $\pi$ and $\pi'$ are both unramified then $\underline{J}=[J_1]$, where $J_1=\{1,\ldots ,m\}$ and $\underline{K}=[K_1]$, where $K_1=\{1,\ldots ,m'\}$. Thus $A=B=L=1$ and $\ell$ sends $(1,1)$ to $1$. Set $\alpha_\pi(p,j)=p^{s_j}$ and $\alpha_{\pi'}(p,k)=p^{s_k'}$, so that $pz_jz_k'=\alpha_\pi(p,j)\alpha_{\pi'}(p,k)$. Then  \eqref{expression} simplifies to the expression \eqref{unramified}.
\end{example}

\subsection{Proof of Proposition \ref{main}}
Let $\mathscr{L}$ denote the image of the injective map $\ell: \Delta\rightarrow \{1,\ldots ,L\}$. Throughout this section we shall write $(a,b)\in\Delta$ for the preimage of $\ell\in\mathscr{L}$. We may rewrite \eqref{expression} as 
\[
L(s,\pi\times\pi')=\prod_{\ell\in \mathscr{L}} L_\ell(s,\pi\times\pi'),
\]
where
\[
L_\ell(s,\pi\times\pi')=\prod_{\nu\geqslant 1} \prod_{\substack{j\in J_a\\ n_j\geqslant \nu}}\prod_{\substack{k\in K_b\\ n_k'\geqslant \nu}} \big(1-(p^\nu z_j z_k')^{e_\ell}p^{-e_\ell s}\big)^{-1}.
\]
Letting $\log L_\ell(s,\pi\times\pi')=\sum_{f\geqslant 1}\frac{\lambda_{\ell,\pi\times\pi'} (f)}{fp^{e_\ell fs}}$, we obtain
\begin{equation}\label{expression-coeff}
\lambda_{\ell,\pi\times\pi'} (f)=\sum_{\nu\geqslant 1} p^{e_\ell\nu f}\bigg(\sum_{\substack{j\in J_a\\ n_j\geqslant \nu}}z_j^{e_\ell f}\bigg)\bigg(\sum_{\substack{k\in K_b\\ n_k'\geqslant \nu}}{z_k'}^{e_\ell f}\bigg).
\end{equation}
\begin{example}
We let $\pi'=\tilde\pi$ and use the notation of Example \ref{first-ex}. Then the identity \eqref{expression-coeff} reduces to
\begin{equation}\label{diagonal-coeffs}
\lambda_{\ell,\pi\times\tilde\pi} (f)=\sum_{\nu\geqslant 1}p^{e_\ell\nu f}\big|\sum_{\substack{j\in F_\ell \\ n_j\geqslant \nu}} z_j^{e_\ell f}\big|^2,
\end{equation}
which recovers the same expression in the proof of \cite[Lemma A.1]{RS}.
\end{example}
\begin{example} When $\pi$ and $\pi'$ are both unramified, formula \eqref{expression-coeff} reduces to
\[
\lambda_{\ell,\pi\times\pi'} (f)=p^{f}\sum_{j=1}^m z_j^f\sum_{k=1}^{m'} {z_k'}^f=\sum_{j=1}^m \sum_{k=1}^{m'} \alpha_\pi(p,j)^f\alpha_{\pi'}(p,k)^f=\lambda_{\pi\times\pi'}(p^f).
\]
\end{example}

Applying the Cauchy-Schwarz inequality  in \eqref{expression-coeff} we get
\begin{equation}\label{ell-bound}
\begin{aligned}
|\lambda_{\ell,\pi\times\pi'}(f)|^2 &\leq \bigg(\sum_{\nu\geqslant 1}p^{e_\ell\nu f} \bigg|\sum_{\substack{j\in J_a\\ n_j\geqslant \nu}}z_j^{e_\ell f}\bigg|^2\biggr) 
\biggl( \sum_{\nu\geqslant 1}p^{e_\ell\nu f} \bigg|\sum_{\substack{k\in K_a\\ n_k'\geqslant \nu}} {z_k'}^{e_\ell f}\bigg|^2\bigg)\\
&= \lambda_{\ell,\pi\times\tilde\pi}(f) \lambda_{\ell,\pi'\times\tilde\pi'}(f),
\end{aligned}
\end{equation}
in view of \eqref{diagonal-coeffs}.

Now from
\begin{align*}
\sum_{f\geqslant 1}p^{-fs}\bigg(\frac{\lambda_{\pi\times\pi'}(p^f)}{f}\bigg)&=\log L(s,\pi\times\pi')\\ 
&=\sum_{\ell\in \mathscr{L}} \log L_\ell(s,\pi\times\pi')\\
&=\sum_{\ell\in\mathscr{L}}\sum_{f\geqslant 1}p^{-e_\ell fs}\bigg(\frac{\lambda_{\ell,\pi\times\pi'}(f)}{f}\bigg)\\
&=\sum_{f\geqslant 1}p^{-fs}\sum_{\substack{\ell\in\mathscr{L}\\ e_\ell\mid f}}\bigg(\frac{\lambda_{\ell,\pi\times\pi'}(f/e_\ell)}{f/e_\ell}\bigg),
\end{align*}
we deduce
\begin{equation}\label{ell-sum}
\lambda_{\pi\times\pi'}(p^f)=\sum_{\substack{\ell\in\mathscr{L}\\ e_\ell\mid f}}e_\ell \lambda_{\ell,\pi\times\pi'}(f/e_\ell).
\end{equation}
Using this and \eqref{ell-bound} we find, by Cauchy-Schwarz, 
\begin{align*}
|\lambda_{\pi\times\pi'}(p^f)|&\leq \sum_{\substack{\ell\in\mathscr{L}\\ e_\ell\mid f}}e_\ell |\lambda_{\ell,\pi\times\pi'}(f/e_\ell)|\\
&\leq \biggl( \sum_{\substack{\ell\in\mathscr{L}\\ e_\ell\mid f}}e_\ell \lambda_{\ell, \pi\times \tilde{\pi}}(f/e_{\ell})\biggr)^{\frac 12}  
\biggl( \sum_{\substack{\ell\in\mathscr{L}\\ e_\ell\mid f}}e_\ell \lambda_{\ell, \pi' \times \tilde{\pi}'}(f/e_{\ell})\biggr)^{\frac 12}  .
\end{align*}
From \eqref{ell-sum} we recognize the right-hand side as $\sqrt{\lambda_{\pi\times\tilde\pi}(p^f)\lambda_{\pi'\times\tilde\pi'}(p^f)}$, proving Proposition \ref{main}.\qed

\bibliographystyle{abbrv}
\bibliography{weak_subconvexEdit.bib}
\end{document}